\documentclass[11pt]{amsart}
\usepackage[utf8]{inputenc}
\usepackage{mathtools}
  \mathtoolsset{centercolon}
\usepackage{mathrsfs}
\usepackage{thmtools}
\usepackage{tikz} \usetikzlibrary{decorations.pathreplacing,arrows.meta,arrows}
\usepackage{tikz-3dplot}
\usepackage{enumitem}
\usepackage{cite}
\usepackage{hyperref}

\theoremstyle{plain}
  \declaretheorem[numberwithin=section]{theorem}
  \declaretheorem[name=Theorem]{alphatheorem}
    
  \declaretheorem[sibling=theorem]{proposition}
  \declaretheorem[sibling=theorem]{lemma}
  \declaretheorem[sibling=theorem]{corollary}
  \declaretheorem[name=Corollary,numbered=no]{corollary*}
  \declaretheorem[numbered=no]{question}
  
  \declaretheorem[numbered=no]{convention}
\theoremstyle{definition}
  \declaretheorem[sibling=theorem]{definition}
  \declaretheorem[sibling=theorem]{example}
\theoremstyle{remark}
  \declaretheorem[numbered=no]{note}

\newcommand{\abs}[1]{\left\lvert #1 \right\rvert}
\newcommand{\vf}{\varphi}
\newcommand{\mb}{\mathbf}
\newcommand{\ms}{\mathscr}
\newcommand{\wt}{\widetilde}
\newcommand{\Xan}{X^{\rm an}}
\newcommand{\Proj}{\mathbf{P}}
\newcommand{\R}{\mathbf{R}}
\newcommand{\circled}[1]{
  \tikz[baseline=(char.base)]{\node[shape=circle,draw,inner sep=1pt] (char) {$#1$};}}

\DeclareMathOperator{\divisor}{div}
\DeclareMathOperator{\Div}{Div}
\DeclareMathOperator{\Prin}{Prin}
\DeclareMathOperator{\dist}{d}
\DeclareMathOperator{\Spec}{Spec}
\DeclareMathOperator{\Trop}{Trop}
\DeclareMathOperator{\trop}{trop}
\DeclareMathOperator{\an}{an}
\DeclareMathOperator{\Log}{Log}
\DeclareMathOperator{\TT}{\mathbf{T}}
\DeclareMathOperator{\G}{\mathbf{G}}
\DeclareMathOperator{\Z}{\mathbf{Z}}

\title[Construction of Fully Faithful Tropicalizations]
{Construction of Fully Faithful Tropicalizations for Curves in Ambient Dimension 3}

\author[T.~Gunn]{Trevor Gunn$^\dagger$}

\address{Trevor Gunn,
School of Mathematics,
Georgia Institute of Technology,
Atlanta,
USA}
\email{tgunn@gatech.edu}

\author[P.~Jell]{Philipp Jell}

\address{Philipp Jell,
Universit\"at Regensburg,
Regensburg, Germany}
\thanks{$^\dagger$Corresponding author.}
\thanks{The second author (PJ) was supported by the DFG Research Fellowship  JE 856/1-1.}

\begin{document}
\allowdisplaybreaks

\begin{abstract}
  In tropical geometry, one studies algebraic curves using combinatorial techniques
  via the tropicalization procedure.
  The tropicalization depends on a map to an algebraic torus
  and the combinatorial methods are most useful when the tropicalization
  has nice properties.
  We construct, for any Mumford curve $X$, a map to a three-dimensional torus,
  such that the tropicalization is isometric to a subgraph of the Berkovich space $\Xan$,
  called the extended skeleton. In this case, we say the tropicalization is ``fully faithful.''

  Additionally, given a map $X$ to a toric variety $Y$, which induces a fully
  faithful tropicalization, we show that we can extend the map to
  $X \to Y \times (\Proj^1)^n$ such that the new tropicalization is smooth and fully faithful.

  \bigskip

  \noindent
  MSC: Primary: 14T05; Secondary: 14G22, 32P05

  \bigskip

  \noindent
  Keywords: Tropical geometry, smooth tropical curves, Mumford curves, extended skeleta,
  faithful tropicalization
\end{abstract}
\maketitle

\section{Introduction}

Classically, it is well-known that while not every algebraic curve is a plane curve,
every curve is a space curve. That is, every curve admits a
closed embedding into $\Proj^3$ (see for instance \cite[Corollary~IV.3.6]{Har}).
Similarly, every graph has an embedding in $\R^3$.
In fact, this can be done with straight lines by putting the vertices as points on the twisted cubic.
Since no plane intersects the twisted cubic in $4$ points, no pair of chords on the twisted cubic can cross.

In this paper, we study the following question, which might be seen as a tropical
combination of these two facts.
\begin{question}
Let $X$ be a Mumford curve over a non-Archimedean field.
Does there exist a map of $X$ to a three-dimensional toric variety such
that the associated tropicalization is fully faithful?
\end{question}
We answer this question positively, with toric variety being $(\mathbf{P}^1)^3$.

\subsection{Fully and totally faithful tropicalizations}
Let us explain the analogy.
Let $Y$ be a toric variety and $X$ an algebraic curve.
Both $X$ and $Y$ have associated Berkovich spaces $\Xan$ and $Y^{\an}$.
The toric variety $Y$ has a canonical tropicalization $\Trop(Y)$
which is a partial compactification of $\R^{\dim Y}$
and comes with a non-constant map $\trop_Y \colon Y^{\an} \to \Trop(Y)$.
For a map from $\varphi \colon X \to Y$
we denote by $\Trop_{\varphi}(X)$ the image of the composition
$\trop_{\varphi} \coloneqq \trop_Y \circ \varphi^{\an} \colon \Xan \to \Trop(Y)$.
We call the space $\Trop_{\varphi}(\Xan)$ an \emph{embedded tropical curve}.
It is canonically equipped with the
structure of a metric graph (potentially with edges of infinite length).

Also associated with $\varphi$ is another metric graph with potentially infinite edges:
the so-called completed extended skeleton $\Sigma = \Sigma(\vf)$,
which is a metric subgraph of $\Xan$.
It was shown by Baker, Payne and Rabinoff \cite{BPR16} that
$\Trop(X) = \trop_\varphi (\Sigma)$
and that $\trop_{\varphi}\!|_{\Sigma} \colon \Sigma \to \Trop(X)$ is a piecewise-linear,
integral affine map of metric graphs.
The tropicalization is called \emph{fully faithful}
if this map is an isometry.
In particular, a fully faithful tropicalization admits a section $\Trop_{\varphi}(X) \to \Xan$.
We can slightly relax those conditions:
A tropicalization is called \emph{totally faithful} if the map is an isometry
when removing the vertices of $\Sigma$ that are infinitely far away.

We prove the following theorem (Theorem~\ref{thm fully faithful}) and
a corollary (Theorem~\ref{thm totally faithful}) that is proved along the way.

\begin{alphatheorem} \label{ThmA}
Let $X$ be a smooth projective Mumford curve.
Then there exist three rational functions $f_1, f_2, f_3$ on $X$ such that
the tropicalization associated to the map $X \to (\mathbf{P}^1)^3, x \mapsto (f_1,f_2,f_3)$
is fully faithful.
\end{alphatheorem}

\begin{corollary*}
Let $Y$ be a proper toric variety of dimension three.
Then there exists a morphism $\varphi \colon X \to Y$
such that the induced tropicalization is totally faithful.
\end{corollary*}

Our construction starts with three piecewise-linear functions on a skeleton of $\Xan$
that were chosen to have the correct combinatorial properties and
then tweaked so that we could lift those piecewise-linear functions to rational functions on $X$.
The choice of these piecewise linear functions
was inspired by Baker's and Rabinoff's construction \cite[Section 8]{BR}.
Here, Baker and Rabinoff construct a faithful tropicalization for any curve in ambient dimension 3.
Since they only consider faithful tropicalizations, they get to fix a skeleton beforehand
(as opposed to a complete extended skeleton) and then construct an embedding that maps
that skeleton isometrically onto its image (as opposed to our situation,
where the completed extended skeleton depends on the embedding).
This means that Baker and Rabinoff get much more freedom when picking their functions and
only require a weaker lifting theorem.

Our main tool is a lifting theorem (Theorem~\ref{lifting thm}) of the second author \cite{J}, that allows
us to lift tropical meromorphic functions on a skeleton
to the algebraic curve $X$.
This theorem refines another lifting theorem of Baker and Rabinoff \cite{BR}.

Similar questions to ours have been considered.
For example in the works of Cartwright, Dudzik, Manjunath, and Yao \cite{CDMY} and
Cheung, Fantini, Park, and Ulirsch \cite{CFPU}.
However, these results are a bit different in spirit, as the authors start
with a given skeleton and then make a construction that works for \emph{some} algebraic curve
with that skeleton.
We also only care about the skeleton of the curve in our construction,
but the map we construct works for \emph{every} curve with that skeleton.

While the main body of our text deals with general Mumford curves,
i.e.\ we do not use any additional properties, our
main technique of lifting tropical meromorphic functions
can also be used to construct nice tropicalizations
for all curves with a given explicit skeleton.
We exhibit this in Section~\ref{sec:genus2} for a special skeleton of genus $2$.

\subsection{Smooth tropicalizations}
We consider another property of tropicalizations: smoothness.
Roughly speaking, an embedded tropical curve is smooth if locally, at every vertex,
the tropical curve looks like the 1-dimensional fan in $\R^k$ whose rays are $e_1,\dots,e_k, -\sum e_i$.

We define in Definition~\ref{def:local-degree-of-non-smoothness} an invariant of an embedded
tropical curve that measures how singular that tropical curve is.
We prove the following resolution of singularities result (Corollary~\ref{Cor resolution of singularities}) by showing
that we can inductively lower this invariant via re-embedding.

\begin{alphatheorem} \label{ThmB}
Let $X$ be a Mumford curve
and $\varphi \colon X \to Y$ a map
that induces a fully faithful tropicalization of $X$.
Then there exist functions $f_1,\dots,f_n$ on $X$
such that $\varphi' \coloneqq \varphi \times (f_1,\dots,f_n) \colon X \to Y \times (\Proj^1)^n$
induces a fully faithful tropicalization of $X$ and such that $\Trop_{\varphi'}(X)$ is
a smooth tropical curve.
\end{alphatheorem}

In Theorem~\ref{thm:resolution}, we prove a resolution procedure for
singularities of embedded tropical curves.
We use this to show that any smooth algebraic curves admits a map to $(\Proj^1)^{2g+2}$
that results in a smooth tropicalization (Corollary~\ref{cor:smooth-tropicalization}).
The best possible bound on the dimension of the ambient space needed is $2g-1$,
since any curve whose minimal skeleton has a vertex of degree $d$ cannot be embedded
smoothly into a space of dimension $2d-2$ (or smaller).
We are hence three off of the optimal bound.

\subsection{Structure}
In Section~\ref{sec:prelim}, we recall the necessary background on tropicalization, Berkovich skeleta
and (tropical) meromorphic functions.

In Section~\ref{sec:construction} we construct three tropical meromorphic functions on the
skeleton, depending on certain parameters, and we show that these functions are liftable.

In Section~\ref{sec:params} we describe conditions on those parameters
that will allow us to prove Theorem~\ref{ThmA}.

In Section~\ref{sec:injective}
we show that if our parameters meet the conditions stated in Section~\ref{sec:params},
the map induced by the lifts of the functions from Section~\ref{sec:construction}
induces a totally faithful tropicalization.

In Section~\ref{sec:compactifications} we complete the proof of Theorem~\ref{ThmA} by showing
that the conditions in Section~\ref{sec:params} can always be met, and
we show that tropicalizations is indeed already fully faithful.

In Section~\ref{sec:res-of-singularities}
we prove Theorem~\ref{ThmB} via a resolution procedure for embedded tropical curves.

In Section~\ref{sec:genus2} we exhibit our lifting techniques on a more specific example
of a genus $2$ skeleton.

\subsection{Acknowledgements}

The authors would like to thank Joe Rabinoff for fruitful discussions and Matt Baker and Walter Gubler
for comments on a preliminary draft.

\section{Preliminaries} \label{sec:prelim}

Throughout this paper, $K$ will denote an algebraically closed field which is complete with respect to a non-trivial,
non-Archimedean absolute value $|\cdot|_K$.
We denote the value group of $K$ by $\Lambda \coloneqq \log |K^\times| \subseteq \mb R$.

\subsection{Tropicalization of curves in \texorpdfstring{$\Proj^n$}{P\textasciicircum n}}
\label{sec:classical-tropicalization}
Most of our work in this paper is concerned with tropicalizing curves in products of projective spaces.
This is a special case of the more general theory of tropicalizing toric varieties as described in Payne's article \cite{Payne}.
Although some results in this paper are phrased in the more general language of toric varieties, it is sufficient for the reader to picture products of projective spaces.

\begin{definition}\label{def:trop-proj}
The \emph{tropical projective space} $\mb{TP}^n$ is the quotient of
\[
  (\R \cup \{-\infty\})^{n + 1} \setminus \{(-\infty, \dots, -\infty)\}
\]
under the following $\R$-action:
\[
  \lambda \cdot (a_0, \dots, a_n) = (a_0 + \lambda, \dots, a_n + \lambda).
\]
\end{definition}

We define a map $\Log \colon \mb{P}_K^n \to \mb{TP}^n$ by
\[
  \Log( [x_0 : \cdots : x_n] ) = [ \log |x_0|_K : \cdots : \log |x_n|_K ]
\]
with the convention that $\log(0) = -\infty$.

\begin{definition}
  When $X$ is a projective variety over $K$ that intersects the torus, $(K^\times)^n$,
  its \emph{tropicalization} is the closure (in the Euclidean topology) of the image of
  $X$ under $\Log$. We denote the tropicalization of $X$ by $\Trop(X)$.
\end{definition}

\subsection{Limits in \texorpdfstring{$\mb{TP}^n$}{TPn}}
Let us look at the simple case of a tropical curve in $\mb{TP}^2$.
This is a piecewise-linear simplicial complex with some set of extreme rays.
Those extreme rays will have a limit point on one of the boundary strata of $\mb{TP}^2$ which we will now describe.

Let $R = \{[0 : a + tu : b + tv]: t \ge 0\}$ be a ray in the affine plane, $\Trop(K^2)$.
Let $\lim R \coloneqq \lim\limits_{t \to \infty} [0 : a + tu : b + tv]$ denote the limit point of this ray.

\begin{enumerate}[label=Case \arabic*.]
  \item If $u < 0$ and $v < 0$ then $\lim R = [0 : -\infty : -\infty]$.
  \item If $u = 0$ and $v < 0$ then $\lim [0 : a : b + tv] = [0 : a : -\infty]$. Similarly if $v = 0$ and $u < 0$.
  \item If $0 \le u < v$ then $[0 : a + tu + b + tv] = [-tv : a + t(u - v) : b]$ and $\lim R = [-\infty : -\infty : 0]$. Similarly if $0 \le v < u$.
  \item If $0 < u = v$ then $[0 : a + tu : b + tu] = [-tu : a : b]$ and $\lim R = [-\infty : a : b]$.
\end{enumerate}

So if $v < u = 0$ or $u < v = 0$ or $0 < u = v$ then the boundary stratum is 1-dimensional.
Otherwise, the boundary stratum is just a single point. Figure~\ref{fig:strata} illustrates this.
For general $n$, the boundary strata of $\mb{TP}^n$ forms a simplex.

\begin{figure}[thbp]
  \centering

  \begin{tikzpicture}
    \foreach \x/\y in {0/0, 4/0, 0/4}
      \filldraw (\x,\y) circle (0.5mm);

    \draw (0,0) -- (4,0) -- (0,4) -- (0,0);

    \draw [dotted] (1.172, 1.172) circle (5mm); % In-center at 4 - 2sqrt(2)

    \foreach \x in {1, 1.344} {
      \filldraw (\x,0) circle (0.5mm);
      \filldraw (0,\x) circle (0.5mm);
      \draw (\x,0) -- (\x,0.703);
      \draw (0,\x) -- (0.703,\x);
    }

    \draw (1.382,1.625) -- (1.89,2.12);
    \draw (1.625,1.382) -- (2.12,1.89);
    \filldraw (1.89,2.12) circle (0.5mm);
    \filldraw (2.12,1.89) circle (0.5mm);
  \end{tikzpicture}

  \caption{Boundary strata of $\mb{TP}^2$. Parallel rays in the directions $(-1,0)$, $(0,-1)$ or $(1,1)$ intersect the boundary in distinct points. Rays in any other direction intersect the closest corner.}
  \label{fig:strata}
\end{figure}

We will see in Section~\ref{sec:injective} that this boundary strata does not have enough components to separate all of our extreme rays. Instead, we will work with $(\mb{TP}^1)^3$.

For $\mb{TP}^1$, Definition~\ref{def:trop-proj} is equivalent to the set $\R \cup \{\pm \infty\}$.
The boundary strata of $(\mb{TP}^1)^n$ can be pictured as the $(n - 1)$-skeleton of an $n$-dimensional cube.
For instance, in $(\mb{TP}^1)^3$, parallel rays in the directions $\pm (0,0,1), \pm (0,1,0), \pm (1,0,0)$ have distinct limits.

\subsection{Metric graphs}\label{sec:metric-graphs}

Let $\Gamma$ be a topological space with a distance function
$\dist \colon \Gamma \times \Gamma \to \mb R \cup \{\infty\}$.
We call $\Gamma$ a \emph{metric graph} if it admits a 1-dimensional simplicial structure where every edge $e$ (aka 1-simplex), with the induced distance function $\dist_e$
is isometric to a closed interval: $[0,l] \subseteq \mb R \cup \{\infty\}$.
We allow the possibility of infinite edges---isometric to $[0,\infty]$---
but we require that these infinite edges be leaf edges.

Explicitly, there exists a set of vertices $V$ and set of edges $E$.
Every edge $e$ has a distance function $\dist_e$, such that $e$ is isometric to a closed interval.
Finally, there are maps $\partial e \to V$ that tell us how to glue the edges to the vertices.
Every edge of $\Gamma$ has the usual distance function
which we extend to $\Gamma$ by setting $\dist(x,y) =$ the length of the shortest path from $x$ to $y$.

A choice of $G = (V, E)$ is called a \emph{graph model} of $\Gamma$. We forget about all the distance functions and topologies on $E$ and just remember the lengths.
In this way, $G$ is a graph where each edge $e \cong [0,l]$ has an associated length $l$.
If $G$ is a graph model of $\Gamma$, then so is any length-respecting subdivision of $G$.
When the graph model is fixed, we may refer to edges and vertices of $G$ as edges and vertices of $\Gamma$.

\begin{note}
  Usually one would call $G$ a ``weighted graph'' but since the term ``weight'' is used in relation to the tropical balancing condition, we avoid this here.
\end{note}

Given a subgroup $\Lambda \subseteq \mb R$ (e.g.\ the value group of $K$),
we say that $\Gamma$ is a \emph{$\Lambda$-metric graph} if it admits a graph model $G = (V, E)$
where the weight of every finite edge of $G$ belongs to $\Lambda$.

Given a graph model $G = (V, E)$ for $\Gamma$,
the \emph{$\Lambda$-rational points} of $\Gamma$ are the points
whose distance to some (and hence every) vertex is an element of $\Lambda$---
we call this set $\Gamma(\Lambda)$.

See Section~2.1 of \cite{ABBR} for another description of a metric graph.

We recall that a \emph{spanning tree} of a (connected) graph is a maximal, acyclic collection of edges
such that every vertex of the graph is an endpoint of one of these edges.
If $e_1,\dots,e_g$ form the complement of such a spanning tree, then $g$---which is well defined---is called the \emph{genus} of $G$.
One can check that if $G$ is a graph model of $\Gamma$ then $g = \dim_{\mb Q} {\rm H}_1(\Gamma;\mb Q)$.

\subsection{Berkovich analytic spaces}

For every variety $X$ over $K$, there is a topological space, $\Xan$, introduced by Berkovich \cite{Berk}
called the \emph{Berkovich analytification}.
The points of $\Xan$ are pairs $(p_x, |\cdot|_x)$ where $p_x \in X$ and $|\cdot|_x$ is an absolute value on
the residue field $k(p_x)$ at the point $p_x$ extending the absolute value of $K$.
The topology on $\Xan$ is the weakest topology making the canonical map $\Xan \to X$ continuous and, for every open set $U$ of $X$ and section $f \in \mathcal{O}_X(U)^\times$, the map $U^{\an} \to \mb R$ given by
\[
(p_x, |\cdot|_x) \mapsto |f(x)| \coloneqq |f(p_x)|_x
\]
is continuous.

\subsubsection{Classification of points}

When $X$ is a curve, the points of $\Xan$ can be classified into four types.

If $p_x$ is a closed point of $X$, then $k(p_x) = K$ and $|\cdot|_x = |\cdot|_K$ is the only absolute value we can take.
In this way, we view $X(K)$ as a canonical subset of $\Xan$. Points in $X(K)$ are called type I points of $\Xan$.

If $p_x$ is the generic point of $X$ and $\ms{H}(x)$ is the completion of $k(p_x)$ with respect to $|\cdot|_x$.
Then we say $(p_x, |\cdot|_x)$ is a type II point if $\operatorname{trdeg}(\wt{\ms H}(x)/\wt{K}) = 1$ where $\wt{\cdot}$ denotes the residue field.

The terminology of type I and type II points is due to Thuiller \cite{Thu} following Berkovich's original classification \cite{Berk}.
There is also a notion of type III and IV points (\emph{loc. cit.}) which we do not make use of in this paper.

\subsection{Skeleta and extended skeleta of curves}
When $X$ is a curve, there exists a distinguished set $\Gamma \subset \Xan$ called a \emph{skeleton of $X$ (or of $\Xan$)} with the following key properties.

\begin{enumerate}
  \item A skeleton is a metric graph.
  \item There is a strong deformation retract $\tau \colon \Xan \to \Gamma$.
  \item The map $\tau_* \colon \Div(X) \to \Div_\Lambda(\Gamma)$ is surjective and takes principal divisors to principal divisors. We define the divisor group of $\Gamma$ in Subsection~\ref{sec:skel-div}.
\end{enumerate}

We start by defining skeletons for open discs and open annuli. More detail is given in \cite[Section 2]{BPR13}.

\begin{definition}\label{def:open-annuli}
  Let $\mb{A}^{1,\an} = (\Spec K[T])^{\an}$. We call the sets
  \[ B(r) \coloneqq \{x \in \mb{A}^{1,\an} : |T|_x < r\} \text{ and } A(r,s) \coloneqq \{x \in \mb{A}^{1,\an} : r < \log |T|_x < s\} \]
  \emph{open discs} and \emph{open annuli} respectively. They are parameterized by real numbers $r, s$ which we call \emph{logarithmic radii}. For an open annulus, we also allow $r = -\infty$ in which case $A(-\infty,s)$ is a \emph{punctured disc}.

  The disc $B(t)$ has a distinguished element $\rho_{B(t)}$ defined by
  \[ \abs{\sum a_i T^i}_{\rho_{B(t)}} = \max\limits_i |a_i| t^i. \]
  As the disc $B(r)$ expands to $B(s)$ in the annulus, we take distinguished elements to form the set
  \[ \Sigma(A(r,s)) \coloneqq \{ \rho_{B(t)} : r < \log t < s \}. \]
  This is called the \emph{skeleton} of $A(r,s)$.
\end{definition}

The annulus $A(r,s)$ canonically retracts onto $\Sigma(A(r,s))$ via
\[ \tau \colon \abs{\cdot}_x \longmapsto \rho_{B(\log |T|_x)}. \]
Berkovich showed that this is a strong deformation retraction \cite[Proposition~4.1.6]{Berk}.

\begin{definition}
For a smooth, projective curve $X/K$, a \emph{semistable vertex set} $V$ of $X$ is a finite set of type II points
in $\Xan$ such that $\Xan \setminus V$ is (isomorphic to) a disjoint union
of finitely many open annuli and infinitely many open discs.
Semistable vertex sets always exist \cite[Proposition~4.22]{BPR13}.
If $\chi(X) \le 0$, then a unique minimal skeleton exists [\emph{loc. cit.}, Corollary~4.23].
\end{definition}

Given a semistable vertex set $V$ of $X$, the associated (finite) \emph{skeleton} is
\[ \Sigma(V) \coloneqq V \bigcup \Sigma(A) \]
where the union is over the finite set of open annuli of $\Xan \setminus V$.
There is a canonical retraction $\tau_V \colon \Xan \to \Sigma(V)$ which is, in fact, a strong deformation retraction.

$\Sigma(V)$ is a $\Lambda$-rational metric graph with a canonical graph model $(V, E)$.
The edges of $\Sigma(V)$ are $\Sigma(A)$ for each open annulus $A$.
The length of the edge $\Sigma(A)$ is the length $s - r$ defined in Definition~\ref{def:open-annuli}.

\subsubsection{Completed skeleta}
A \emph{completed semistable vertex set} is defined the same as a semistable vertex set except
we also allow ourselves to include some points of type I.
These type I points are infinitely far away from the finite skeleton.
If $V$ is a completed semistable vertex set,
then the set of type II points in $V$ form a semistable vertex set by themselves.

The skeleton associated to a completed semistable vertex set is called a \emph{completed skeleton}.
It is defined similarly.
The main difference is that the addition of type I points turns
some open discs of $\Xan \setminus V$ into punctured discs.
The skeleton of a punctured disc is an edge of infinite length.

\begin{convention}
We typically use the letter $\Gamma$ in this paper for a finite skeleton and $\Sigma$ for a completed skeleton.
\end{convention}

\subsubsection{Skeleta associated to toric embeddings}

Let $X$ be a smooth projective curve and let $\vf \colon X \to Y$
be a closed embedding of $X$ into a toric variety $Y$.
Let $T$ be the dense torus in $Y$. Let $X^\circ = \vf^{-1}(T)$.

\begin{definition}
The \emph{completed extended skeleton associated to $\vf$} is the set $\Sigma(\vf)$ of points in $\Xan$
that do not have an open neighborhood contained in $(X^\circ)^{\an}$ and isomorphic to an open disc. We write $\mathring{\Sigma}(\vf)$ for the skeleton $\Sigma(\vf)$ with its type I points removed.
\end{definition}

\begin{example}
If $Y$ is a product of $\mathbf{P}^1$'s,
then $\vf$ is defined by a set of rational functions and $X^\circ$ is the set of points
that are neither zeroes nor poles of those functions.
The skeleton $\Sigma(\vf)$ contains all of those zeroes and poles as type I points.
\end{example}

\subsection{Tropicalization of analytic curves}
If $Y$ is a projective space (or product of projective spaces) over $K$,
then the map $\Log \colon Y \to \Trop(Y)$ defined in Section~\ref{sec:classical-tropicalization}
extends to the analytification, $Y^{\an}$.
We call this map $\trop \colon Y^{\an} \to \Trop(Y)$.

More generally, if $Y$ is a toric variety, then there is a map $\trop \colon Y^{\an} \to \Trop(Y)$.
See \cite[Section 3]{Payne} for the definition.

\begin{example}
When $Y = \mb P^1 = \operatorname{Proj}K[z_0, z_1]$,
the map $\trop \colon \mb P^{1} \to \mb{TP}^1$ is given by
\[ \trop((p, |\cdot|_x)) = \log|z_1(p)|_x. \qedhere\]
\end{example}

When there is a closed embedding $\vf$ of $X$ into the toric variety $Y$ (e.g.\ if $X$ is projective),
we can use this to tropicalize $X$ via
\[ \trop_{\vf} \coloneqq \trop \circ \varphi^{\an} \colon \Xan \to \Trop(Y). \]
The image of $\Xan$ under $\trop_{\vf}$ is denoted $\Trop_{\vf}(X)$.

\subsection{Fully faithful, totally faithful and smooth} \label{sec:ff-and-smooth}

Let $\varphi  \colon X \to Y$ be a map from $X$ to a toric variety $Y$,
that is generically finite and whose image meets the dense torus $T$ of $Y$.
Let $U \coloneqq \varphi^{-1}(T).$
Let $N$ be the cocharacter lattice of $T$ and $N_\R \coloneqq N \otimes_{\Z} \R$.
The map $\trop_{\varphi}$
is called \emph{totally faithful} (see \cite{CFPU}) if it induces an isometry from the associated
open skeleton $\mathring \Sigma(\varphi)$ onto its image
(which is exactly $\trop(\Xan) \cap N_\R$.)
It is called \emph{fully faithful} if it is further injective when restricted to $\Sigma(\varphi)$.
This is equivalent to the statement that $\trop_{\varphi}$ is injective when restricted to $\varphi^{-1}(Y \setminus T)$.

The map $\trop_{\varphi}\!|_{\Sigma(\varphi)}$ is linear with integral slope on each edge of $\Sigma(\varphi)$.
We call this slope the \emph{stretching factor} of $\trop_{\varphi}$ on $e$.
Identifying $T$ with $\G_m^n$, the restriction $\varphi_{U}$ is given by rational functions $f_1,\dots,f_n$ on $X$.
Then the stretching factors of $\trop_{\varphi}$ on $e$ is given by the gcd of the slopes of
$\log \vert f_i \vert|_e$, $i = 1,\dots,n$ \cite[5.6.1]{BPR16}.
In particular, $\varphi$ induces a fully faithful tropicalization if $\trop_{\varphi}\!|_{\Sigma(\varphi)}$
is injective and all stretching factors are equal to one.

Let $\varphi \colon X  \to Y$ be a closed embedding and let $\Sigma(\varphi)$
be the associated completed extended skeleton.
We say that $\trop_{\varphi}$ is a \emph{smooth tropicalization}
if it is fully faithful and further
for every finite vertex $x$ of $\Sigma(\varphi)$ the
primitive integral vectors along the edges adjacent $\trop_{\varphi}(x)$
span a saturated lattice in $N$ of rank $\deg(x) - 1$.

Usually the conditions for smoothness for tropical curves
do not reference fully faithfulness and instead weights.
This is equivalent to our definition in view of \cite[Section 5]{J}.

\subsection{Divisors and rational functions on a metric graph}\label{sec:skel-div}

If $\Gamma$ is a $\Lambda$-metric graph then a ($\Lambda$-rational) \emph{divisor} on $\Gamma$ is a finite,
formal integer-linear combination of $\Lambda$-rational points on $\Gamma$.
These divisors form a free Abelian group, which we call $\Div_\Lambda(\Gamma)$.

A \emph{rational function} on $\Gamma$ is a piecewise-linear function $F$ with integer slopes and
such that all the points where $F$ is non-linear are $\Lambda$-rational.
If these points where $F$ is non-linear are called $x_1,\dots,x_n$,
then the \emph{principal divisor} associated to $F$ is
\[ \sum_{i = 1}^n m_i x_i \]
where $m_i$ is the sum of the outgoing slopes of $F$ at $x_i$.
The principal divisors on $\Gamma$ form a subgroup, which we call $\Prin_\Lambda(\Gamma)$.

If $\tau \colon \Xan \to \Gamma$ is the deformation retraction of $\Xan$ onto its skeleton,
then $\tau$ maps $X(K)$ onto $\Gamma(\Lambda)$.
We can therefore extend this map to a surjective map $\tau_* \colon \Div(X) \to \Div_\Lambda(\Gamma)$.

Let $f \in K(X)^*$ be a rational function.
Then $\log|f|$ is a function on $\Xan$.
If $F$ is the restriction of $\log|f|$ to $\Gamma$, then it is known that $F$ is a $\Lambda$-rational function.
Moreover, \[ \tau_*\divisor(f) = \divisor(F). \]
This means that $\tau_*$ takes principal divisors to principal divisors.

\begin{note}
These two facts about $\log|f|$ are referred to as the ``slope formula'' or
``non-Archimedean Poincaré-Lelong formula'' in the literature.
The formula was first stated and proved in our terminology by Baker, Payne and Rabinoff~\cite{BPR13},
Theorem~5.15.
The original result is due to Thuiller~\cite{Thu} who phrased it in terms of potential theory.
Thuiller's formulation closely resembles the classical formula for complex manifolds.
\end{note}

More results about the connection between $\Div(X)$ and $\Div_\Lambda(\Gamma)$
may be found in \cite{Baker} and \cite{BR}.

\begin{definition}
An effective divisor $B$ on a metric graph $\Gamma$ is called a \emph{break divisor}
if there exists a graph model $G$ of $\Gamma$ and edges $e_1,\dots,e_g$ of $G$
forming the complement of a spanning tree
such that $B = x_1 + \dots + x_g$ where $x_i \in e_i$.
\end{definition}

Break divisors were first introduced by Mikhalkin and Zharkov \cite{MZ} and
were used by An, Baker, Kuperberg, and Shokrieh \cite{ABKS}
to give a geometric proof of Kirchhoff's Matrix-Tree Theorem.

\subsection{Mumford curves}

\begin{definition}
  A smooth, projective curve $X$ over $K$ is called a \emph{Mumford curve}
if the genus of $X$ is equal to the genus (i.e.\ the first Betti number) of its skeleton.
\end{definition}

While the question of which curves admit fully or totally faithful tropicalizations
is still open, it is known that only Mumford curves admit smooth tropicalizations.

\begin{theorem}{\rm \cite[Theorem A]{J}}
Let $X$ be a smooth projective curve.
Then the following are equivalent
\begin{enumerate}
\item
$X$ is a Mumford curve.
\item
There exists an embedding $\varphi \colon X \to Y$ for a toric variety $Y$
such that $\Trop_{\varphi}(X)$ is smooth.
\end{enumerate}
\end{theorem}

This theorem shows that, at least for the results of Section \ref{sec:res-of-singularities}, we have to consider Mumford curves.
The question of whether general smooth algebraic curves admit
fully faithful tropicalizations is open for non-Mumford curves.

\section{Construction of fully faithful tropicalization in 3-space} \label{sec:construction}

In this section, $X$ will denote a Mumford curve over a complete, algebraically closed,
non-Archimedean valued field $K$ with analytification $\Xan$ and skeleton $\Gamma$.
We take $G$ to be a graph model of $\Gamma$ with vertex set $V = V(G)$ and edge set $E = E(G)$.

After possibly subdividing, we assume that $G$ has edges $e_1,\dots,e_g$ that form the complement of a spanning tree,
$T \subseteq E$, and that no two edges $e_i, e_j$ share a vertex.

We will define three piecewise-linear functions $F_1, F_2, F_3$ on $\Gamma$
whose graphs are depicted in Figures~\ref{fig:F1-graph} to \ref{fig:F2-graphT}.
To construct these piecewise-linear functions, we consider divisors on $\Gamma$ and use the following lifting theorem.

\begin{theorem}[Jell] \label{lifting thm}
Let $D$ be a divisor on $X$ of degree $g$.
Given any break divisor $B = x_1 + \dots + x_g$ on $\Gamma$ supported on 2-valent points,
if $\tau_*D - B$ is principal then there exist liftings $x_1',\dots,x_g' \in X(K)$ such that $\tau_*x_i' = x_i$ and
such that $D - \sum_{i = 1}^g x_i'$ is a principal divisor.
\end{theorem}

\begin{proof}
Theorem 3.2 of \cite{J}.
\end{proof}

Another equivalent way of writing this theorem is the following.

\begin{theorem} \label{lifting thm convenient}
Let $D = \sum_{i=1}^k a_i - \sum_{j = 1}^k b_j$ be a principal divisor on $\Gamma$.
Assume that $\sum_{i=1}^g a_i$ is a break divisor supported on 2-valent points.
Then, given preimages $x_i$ and $y_j$ for all $i = g+1,\dots,k$ and all $j = 1,\dots,k$
such that $\tau(x_i) = a_i$ and $\tau(y_j) = b_j$,
there exist $x_1,\dots,x_g \in X(K)$ with $\tau(x_i) = a_i$ such that
$\sum_{i=1}^k x_i - \sum_{j=1}^k  y_i$ is a principal divisor on $X$.
\end{theorem}
\begin{proof}
This follows from the lifting theorem applied with
\[
D = \sum_{i = 1}^{k} b_i - \sum_{i = g+1}^k a_i \text{ and }  B = \sum_{i=1}^{g} x_i.
\qedhere
\]
\end{proof}

\subsection{Constructions of the piecewise-linear functions and lifting}

We construct the piecewise-linear functions $F_1$, $F_2$ and $F_3$ by specifying their divisors.
To construct these divisors, we will need to choose, for each edge $e$, points
which will be labeled $c_e, a_e, p_e, q_e, b_e, d_e$ in the interior of $e$. This will be the order of the points in their respective edge.
We also require that the pairs $c_e, d_e$ and $a_e, b_e$ and $p_e, q_e$ are symmetric about the middle of their edges.

We will describe the exact position of these points inside their edges in Section~\ref{sec:params}.
The statements of this section do not depend on the choices made in Section~\ref{sec:params}.

We pick the following additional data:
For every edge $e$, we label one of its endpoints $v(e)$ and the other one $w(e)$
and we pick for each edge $e$ a positive integer $s(e)$.
We will describe which vertex is $v(e)$ and which is $w(e)$ in
Section~\ref{sec:params} along with conditions for the integers $s(e)$.

Let $\{e_1, \dots, e_g\}$ be the edges not in the spanning tree $T$ and
note that the following divisors are all principal
\begin{align*}
D_1 &= \sum_{e \in E} v(e) + w(e) - p_e - q_e, \\
D_2 &= \sum_{e \in E} s(e) \left( v(e) + w(e) - p_e - q_e \right) +
\sum_{i = 1}^g -c_{e_i} + a_{e_i} + b_{e_i} - d_{e_i}, \\
D_3 &= \sum_{e \in E} a_e - b_e.
\end{align*}
Let $F_i$ be a piecewise-linear function such that $\divisor(F_i) = D_i$.
The graphs of $F_i$ are depicted in Figures \ref{fig:F1-graph}, \ref{fig:F3-graph}, \ref{fig:F2-graphT'}
and \ref{fig:F2-graphT}.
Our graphs look similar to the graphs of the functions used by Baker and Rabinoff
(and depicted in \cite[Figure 1]{BR}), however they are tweaked to fit with our lifting theorem.
Notice for example the slight bumps in Figure \ref{fig:F2-graphT},
which are there specifically to allow application of our lifting theorem.

%%%% Graphs %%%%

\begin{figure}[hbp]

\centering
%% F1 %%
\begin{minipage}{0.49\textwidth}
\begin{tikzpicture}[scale=0.6]
  %the base line
  \draw  (1,-1) -- (8,-1);

  %the graph
  \draw (1,0) -- (4,3) -- (5,3) -- (8,0);

  \node[below] at (1,-1) {$v(e)$};
  %\node[below] at (2,-1) {$c_e$};
  %\node[below] at (3,-1) {$a_e$};
  \node[below] at (4,-1) {$p_e$};
  \node[below] at (5,-1) {$q_e$};
  %\node[below] at (6,-1) {$b_e$};
  %\node[below] at (7,-1) {$d_e$};
  \node[below] at (8,-1) {$w(e)$};
  %\node[below] at (5,5) {\phantom{a}};

  \node[rotate=45, above] at (2.5,1.5) {slope 1};
\end{tikzpicture}
\caption{The graph of $F_1|_{e}$.}
\label{fig:F1-graph}
\end{minipage}
%% F3 %%
\begin{minipage}{0.49\textwidth}
\begin{tikzpicture}[scale=0.6]
  %the base line
  \draw  (1,-1) -- (8,-1);

  %the graph
  \draw (1,0) -- (3,0) -- (6,3) -- (8,3);

  \node[below] at (1,-1) {$v(e)$};
  %\node[below] at (2,-1) {$c_e$};
  \node[below] at (3,-1) {$a_e$};
  %\node[below] at (4,-1) {$p_e$};
  %\node[below] at (5,-1) {$q_e$};
  \node[below] at (6,-1) {$b_e$};
  %\node[below] at (7,-1) {$d_e$};
  \node[below] at (8,-1) {$w(e)$};

  \node[rotate=45, above] at (4.5, 1.5) {slope 1};

  \node[above right] at (1,0) {$r(v(e))$};
  \node[above left] at (8,3) {$r(w(e))$};

  \end{tikzpicture}
\caption{The graph of $F_3|_{e}$.}
\label{fig:F3-graph}
\end{minipage}

\vspace{.5cm}

%% F2T %%
\begin{minipage}{0.49\textwidth}
  \begin{tikzpicture}[scale=0.6]
    %the base line
    \draw  (1,-1) -- (8,-1);

    %the graph
    \draw (1,0) -- (4,5) -- (5,5) -- (8,0);

    \node[below] at (1,-1) {$v(e)$};
    %\node[below] at (2,-1) {$c_e$};
    %\node[below] at (3,-1) {$a_e$};
    \node[below] at (4,-1) {$p_e$};
    \node[below] at (5,-1) {$q_e$};
    %\node[below] at (6,-1) {$b_e$};
    %\node[below] at (7,-1) {$d_e$};
    \node[below] at (8,-1) {$w(e)$};

    \node[rotate=59.03, above] at (2.5,2.5) {slope $s(e)$};

    \end{tikzpicture}
\caption{The graph of $F_2|_{e}$ for $e \in T$.}
\label{fig:F2-graphT'}
\end{minipage}
%% F2T' %%
\begin{minipage}{0.49\textwidth}
\begin{tikzpicture}[scale=0.6]
  %the base line
  \draw  (1,-1) -- (8,-1);

  %the graph
  \draw (1,0) -- (2,2) -- (3,3) -- (4,5) -- (5,5) -- (6,3) -- (7,2)  -- (8,0);

  \node[below] at (1,-1) {$v(e)$};
  \node[below] at (2,-1) {$c_e$};
  \node[below] at (3,-1) {$a_e$};
  \node[below] at (4,-1) {$p_e$};
  \node[below] at (5,-1) {$q_e$};
  \node[below] at (6,-1) {$b_e$};
  \node[below] at (7,-1) {$d_e$};
  \node[below] at (8,-1) {$w(e)$};

  \node[rotate=63.44, above] at (1.5,1) {slope $s(e)$};
  \node[rotate=-45,   below] at (6.5,2.5) {slope $s(e) - 1$};
  \node[rotate=63.44, above] at (3.5,4) {slope $s(e)$};
\end{tikzpicture}
\caption{The graph of $F_2|_{e}$ for $e \notin T$.}
\label{fig:F2-graphT}
\end{minipage}

\end{figure}

%%%%%%%%%%%%%%%%

We now want to lift these functions to $\Xan$ by lifting their divisors using Theorem~\ref{lifting thm convenient}.

\begin{proposition} \label{D3 lift}
For every $e$ there exist lifts $a'_e, b'_e \in X(K)$ of $a_e, b_e$ such that
\begin{align*}
D'_3 \coloneqq \sum_{e \in E} a'_e - b'_e
\end{align*}
is a principal divisor on X.
\end{proposition}

\begin{proposition} \label{D1 lift}
For every point in $\{ v(e), w(e), p_e, q_e \mid e \in E \}$
there exist a lift in $X(K)$, which we denote by $v(e)', w(e)', p'_e, q'_e$ respectively
such that
\begin{align*}
D'_1 \coloneqq \sum_{e \in E} v(e)' + w(e)' - p'_e - q'_e
\end{align*}
is a principal divisor on $X$.
\end{proposition}

\begin{note}
In the previous two propositions, we did not prescribe any lifts for the points
in the support of $D_3$ or $D_1$.
However, in the lifting theorem allows us to prescribe all but $g$ lifts.
In the following proposition we will do just that,
using the full power of Theorem \ref{lifting thm convenient}.
\end{note}

\begin{proposition} \label{D2 lift}
Suppose that for every point in $\{a_e, b_e, v(e), w(e), p_e, q_e \mid e \in E \}$,
we are given lifts $a'_e, b'_e, v(e)', w(e)', p'_e, q'_e \in X(K)$ respectively.
Then for every $i = 1,\dots,g$, there exist lifts $c'_{e_i}$ and $d'_{e_i}$
of $c_{e_i}$ and $d_{e_i}$ such that
\begin{align*}
D_2' \coloneqq \sum_{e \in E} s(e) \left( v(e)' + w(e)' - p_e' - q_e' \right) +
\sum_{i = 1}^g -c_{e_i}' + a_{e_i}' + b_{e_i}' - d_{e_i}'
\end{align*}
is a principal divisor on $X$.
\end{proposition}
\begin{proof}
All three Propositions follow directly from Theorem~\ref{lifting thm convenient}.
\end{proof}

We let $f_1, f_2, f_3 \in K(X)$ be such that $\divisor(f_i) = D'_i$ so that $\log \vert f_i \vert|_{\Gamma} = D_i$.
Let $U$ be the open set of $X$ obtained by removing all the points
$v'(e)$, $w'(e)$, $a_e'$, $b_e'$, $c_e'$, $d_e'$, $p_e'$, $q_e'$ for each edge $e$.
Then we have the map
\begin{align*}
f \coloneqq (f_1, f_2, f_3) \colon U \to \mb G_m^3.
\end{align*}
For every three-dimensional, proper toric variety $Y$, this map extends
to a morphism
\begin{align*}
\varphi \colon X \to Y.
\end{align*}

\begin{proposition} \label{prop inj implies faithful}
Assume that for a vertex $v$ of $\Sigma(\varphi)$,
the number of adjacent edges is coprime to $\sum_{e : v \in e} s(e)$ and
that $\trop_{\varphi}\!|_{\mathring \Sigma(\varphi)}$ is injective.
Then the tropicalization induced by $\varphi$ is totally faithful.

Similarly, if $\trop_{\varphi}\!|_{\Sigma(\varphi)}$ is injective, then
the tropicalization induced by $\varphi$ is fully faithful.
\end{proposition}
\begin{proof}
We have to check that for each domain of linearity of the functions $\log \vert f_i \vert$,
the gcd of their slopes is equal to $1$.
The extended skeleton $\Sigma$ associated to $\varphi$ is given by taking $\Gamma$
and at each point $c_e, a_e, p_e, q_e, b_e, d_e$ adding a ray $[c_e, c'_e)$ and so on.
Note that here it is crucial that we were able to select the points we obtained in Proposition~\ref{D3 lift} and
Proposition~\ref{D1 lift} and reuse them in Proposition~\ref{D2 lift},
otherwise we would have to potentially add multiple edges.

On the finite edges we have $\log \vert f_i \vert = F_i$, so this can be checked directly
(c.f. Figures \ref{fig:F1-graph}, \ref{fig:F3-graph}, \ref{fig:F2-graphT'}
and \ref{fig:F2-graphT}.
).

On an infinite edge, $e$, the slope of $\log \vert f_i \vert$ is the coefficient of $D_i$ at
the finite endpoint of $e$.
So again this can be checked case by case.
\end{proof}

\section{The right choice of parameters}\label{sec:params}

We now describe conditions on the parameters for which, as we will show in the next section,
the tropicalization map induced by $(f_1,f_2,f_3)$ will be fully faithful.

By parameters, we mean: a subdivision of the skeleton $\Gamma$ of $\Xan$ that is suitable,
the distance of the points $c_e, a_e, p_e, q_e, b_e$ and $d_e$ from the vertices
as well as the values $r(v)$ for each vertex $v$ and $s(e)$ for each edge $e$.

\subsection{Interval condition}\label{sec:interval-cond}
Except for the symmetry of the pairs $c_e, d_e$, $a_e, b_e$ and $p_e, q_e$ about their edge's midpoint,
we have complete freedom on where we choose these points on the interior of each edge.
The arrangement of these points is pictured in Figure~\ref{fig:interval} which we will now describe.

Map each edge $e$ to the real line so that it has one of its vertices, $v(e)$,
at $0$ and the other vertex, $w(e)$ at $\ell(e) = $ the length of $e$.

Then, we require that the points $v(e), c_e, a_e, p_e$
can be grouped into disjoint intervals according to what kind of point they are.
Namely, every point $c_e$ should lie to the left of any point $a_{e'}$, should lie to the left of any point $p_{e''}$. The most restrictive requirement is that we want a point $p_e$ to be to the left of the midpoint of any other edge.

We require that symmetric conditions hold if all the edges are right-aligned at their vertex $w(e)$. That is, $q_e$ should be to the right of every midpoint and every point $b_{e'}$ should be to the right of $q_e$ and every point $d_{e''}$ should be to the right of $b_{e'}$.

We will call this requirement on the arrangement of the points, the \emph{interval condition}.

\begin{figure}[htbp]
  \centering
  \begin{tikzpicture}[scale=0.95]
    \draw (-0.3,0) -- (12.3,0);
    \draw (0,0.25) -- (0,-0.25) node[below] {$v(e)$};

    \draw(-0.3,-1) -- (12.3,-1);
    \draw (12,-0.75) -- (12,-1.25) node[below] {$w(e)$};

    \draw (5.8,0.25) rectangle (6.2,-1.25);
    \node[above] at (6,0.25) {midpoints};

    \foreach \x/\xtext in {1/$c_e$, 3/$a_e$, 5/$p_e$} {
      \draw [{Arc Barb[width=5mm]}-{Arc Barb[width=5mm]}] (\x-0.6,0) -- (\x+0.6,0) {};
      \node[below] at (\x,-0.25) {\xtext};
    }

    \foreach \x/\xtext in {7/$q_e$, 9/$b_e$, 11/$d_e$} {
      \draw [{Arc Barb[width=5mm]}-{Arc Barb[width=5mm]}] (\x-0.6,-1) -- (\x+0.6,-1) {};
      \node[below] at (\x,-1.25) {\xtext};
    }
  \end{tikzpicture}
  \caption{Where the points lie on the real line.}
  \label{fig:interval}
\end{figure}
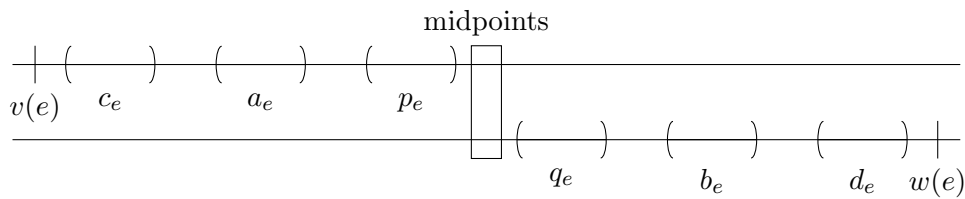

\subsection{Conditions for \texorpdfstring{$r(v)$}{r(v)}}

We now describe conditions for the constants $r(v)$ that will be the values of $F_3$ at the vertices $v$
(i.e.\ $r(v) = F_3(v)$). These constants are related to the points $a_e$ and $b_e$ by
\begin{align*}
 \dist_e(a_e, b_e) = \abs{r(w) - r(v)}
\end{align*}
for an edge $e = vw$.

As such, we require that $\abs{r(w) - r(v)}$ is strictly smaller than the length of $vw$.
By convention, we will write $v(e)$ for the vertex of $e$ with the smaller value of $r$
and $w(e)$ for the larger value.

We also require two additional properties for the values of $r$:

\begin{enumerate}[label=(R\arabic*)]
  \item \label{R:distinct} $r(v)$ is distinct for each $v \in V(G)$.
  \item \label{R:A-distinct} The distances $\dist(a_e,v(e)) = \dist(b_e,w(e)) = F_1(a_e)$ are distinct for each $e \in E(G)$.
\end{enumerate}

\subsection{Further requirements on locations} \label{sec:distinctness}
In addition to having distinct values of $F_1$ for $a_e$, we require the following conditions:

\begin{itemize}
  \item for each edge $e \notin T$, the points $c_{e}$ to be chosen such that the distances
  $\dist_{e}(v(e), c_{e})) = F_1(c_e)$ are all distinct,
  \item and, for each edge $e$, we require the points $p_e$ to be chosen
  such that the values of $F_3(p_e) = r(v(e)) + \dist_e(p_e, a_e)$ are all distinct,
  \item and, for each edge $e$, we require that the points $q_e$ are chosen such that the values of $F_3(q_e) = r(w(e)) - \dist_e(q_e,b_e)$ are all distinct,
  \item and finally, we require that $F_3(p_e) \ne F_3(q_{e'})$ for any $e, e' \in E$.
\end{itemize}

\begin{note}
  These conditions do not impose a significant restriction because: the points are to be chosen from an interval, the $\Lambda$-rational points are dense,
  and there are only finitely many choices to avoid.
\end{note}

\begin{definition}
For each edge $e$, $\vf_e \colon e \to [0,\ell(e)]$ will denote the isometry with $\vf_e(v(e)) = 0$ and
$\vf_e(w(e)) = \ell(e)$. If $x \in \Gamma$ is not a vertex then it is contained in a unique edge $e$, and we will write $\vf(x)$ for $\vf_e(x)$.
\end{definition}

\subsection{Conditions for \texorpdfstring{$s(e)$}{s(e)}}

Recall that to define $F_2$ we have to choose for each edge, $e$, an integer $s(e) > 1$.
We require that these integers satisfy the following conditions
\begin{enumerate}[label=(S\arabic*)]
  \item \label{S:distinct} For every edge $e$, the integers $s(e)$ are all distinct.
  \item \label{S:hordistinct} For every edge $e$, the value of $F_2$ on the interval $[p_e, q_e]$, is distinct.
  \item \label{S:c} For any $e \in T$, $e' \notin T$ and any $x \in e$ we have $F_2(x) < F_2(c_{e'})$.
    Furthermore, the distance between $F_2(p_e) = \max F_2|_e$ and $F_2(c_{e'})$ exceeds (strictly)
    \[
      \max_{y \in \Gamma} F_3(y) - \min_{y \in \Gamma} F_3(y) = \max_{v \in V(G)} r(v) - \min_{v \in V(G)} r(v).
    \]
  \item \label{S:disjoint} For every edge $e \notin T$,
    the intervals $[F_2(c_{e}), F_2(p_{e})] \subseteq \mb R$ are disjoint.
    Again, the distance between these intervals should be large in the same sense as \ref{S:c}.
    Namely, if $F_2(p_{e}) < F_2(c_{e'})$ for a different edge $e' \notin T$ then
    \[
      F_2(c_{e'}) - F_2(p_{e}) > \max_{y \in \Gamma} F_3(y) - \min_{y \in \Gamma} F_3(y).
    \]

\end{enumerate}

\begin{note}
  Figure~\ref{fig:ae-ray-gap} on page~\pageref{fig:ae-ray-gap} shows what \ref{S:c} and \ref{S:disjoint} are designed to accomplish.
\end{note}

\begin{enumerate}[label=(S\arabic*)] \setcounter{enumi}{4}
  \item \label{S:F2lambda} For all $e \in T$ and $e' \notin T$. If $x \in e'$ with $\vf(p_e) \le \vf(x) \le \vf(q_e)$ then
  $F_2(x) > F_2(p_e) + s(e)\lambda$ for any $\lambda \le \max F_3 - \min F_3$.
\end{enumerate}

\begin{note}
  The idea is that $F_2(x) \approx F_2(p_{e'})$ and \[ F_2(p_{e'}) \approx s(e') F_1(p_{e'}) \gg s(e)F_1(p_e) = F_2(p_e). \]
  This is to get around the fact that $F_2|_{e'}$ is not simply equal to $s(e')F_1|_{e'}$ as is the case in the construction of Baker and Rabinoff~\cite[Theorem 8.2]{BR}.
\end{note}

\begin{enumerate}[label=(S\arabic*)] \setcounter{enumi}{5}
  \item \label{S:coprime} For each $v \in V$, $\deg(v)$ is coprime to $\sum_{e \ni v} s(e)$.
\end{enumerate}

\section{Injectivity} \label{sec:injective}

In this section we continue with the notation from the previous section.
Let $X$ be a Mumford curve with a finite skeleton $\Gamma$ and a graph model $(V, E)$,
Assume that for each $e \in E$, we have chosen points
$c_e,a_e, p_e, q_e, b_e, d_e$ satisfying the interval condition.
Let $Y$ be a proper toric variety of dimension $3$, and
$\varphi \colon X \to Y$ the morphism that
is, on the dense torus, given by the functions $f_1,f_2,f_3$ constructed in Section \ref{sec:construction}.

Again, $F_1, F_2, F_3$ are piecewise linear functions with $F_i = \log |f_i|$. For convenience, we will choose $F_1$ and $F_2$ to take the value $0$ at any vertex in $V$.

\begin{proposition} \label{prop choosing parameters}
Let points be chosen on each edge satisfying the interval condition.
Choose parameters $r(v)$ and $s(e)$ satisfying \ref{R:distinct} and \ref{R:A-distinct} and
\ref{S:distinct}--\ref{S:coprime}.
Then the map $\trop_{\varphi}\!|_{\mathring \Sigma} \colon \mathring \Sigma \to  \mb R^3$ is injective.
\end{proposition}

The proof of this proposition is broken up into several lemmas. In each, we assume the conditions of Proposition~\ref{prop choosing parameters} hold.

\begin{lemma} \label{lem:F1-F2}
Suppose that $x, y \in \Gamma \setminus V$ such that $F_1(x) = F_1(y)$ and $F_2(x) = F_2(y)$.
Then $x$ and $y$ are contained in the same edge $e$ of $\Gamma$ and one of the following holds
\begin{enumerate}
\item $x = y$,
\item $x$ is the reflection of $y$ about the middle of $e$,
\item $x, y \in [p_e, q_e]$.
\end{enumerate}
\end{lemma}

\begin{proof}
By reflecting $x$ or $y$ about the middle of their respective edges $e_1$ and $e_2$ if necessary,
we may assume that $v(e_1)$ and $v(e_2)$ are the respective closest vertices.
Further, if $x$ is contained in $[p_{e_1}, q_{e_1}]$, we may replace it by $p_{e_1}$ and the same
goes for $y$ and $p_{e_2}$.

Now we have to show that after these replacements, we have $x = y$.
First observe that \ref{S:disjoint} and \ref{S:F2lambda} imply that if at least one of $e_1, e_2$ is not in $T$, then $F_2(x) = F_2(y)$ imply that either $e_1 = e_2$ (in which case $F_1(x) = F_1(y)$ implies $x = y$) or both $\varphi(x) < \varphi(c_{e_1})$ and $\varphi(y) < \varphi(c_{e_2})$---which is the interval on which $F_2|_e = s(e)F_1|_e$ regardless of whether $e \in T$ or not.

And now we have \[ \dist_{e_1}(v(e_1), x) = F_1(x) = F_1(y) = \dist_{e_2}(v(e_2), x) \] and
\[ s(e_1) \dist_{e_1}(v(e_1), x) = F_2(x) = F_2(y) = s(e_2) \dist_{e_2}(v(e_2), x). \]
It follows from these equations that $s(e_1) = s(e_2)$ and thus $e_1 = e_2$.
Then the first equation implies $x = y$.
\end{proof}

\begin{lemma}   \label{lem:injective-on-Gamma}
The map $F|_{\Gamma} \colon \Gamma \to \mb R^3$ is injective.
\end{lemma}

\begin{proof}
Suppose $x, y \in \Gamma$ and $F(x) = F(y)$.
If $F_1(x) = F_1(y) = 0$
then $x$ and $y$ are vertices and so $r(x) = F_3(x) = F_3(y) = r(y)$.
Since $r$ takes distinct values on distinct vertices, this means $x = y$.

Otherwise, if $F_1(x) = F_1(y) \neq 0$ then $x$ and $y$ are not vertices.
It now follows from Lemma~\ref{lem:F1-F2} that $x$ and $y$ lie on the same edge.
If $x, y \in [p_e,q_e]$ then $x = y$ since $F_3|_{[p_e,q_e]}$ is injective.
Otherwise, Lemma~\ref{lem:F1-F2} gives us that $x = x'$ or $x$ is $y$ reflected about the midpoint of its edge.
On the other hand, $F_3$ is antisymmetric on each edge so $F_3(x) = F_3(y)$ means that $x = y$.
\end{proof}

\subsection{Infinite rays}

\begin{table}[htbp]
  \small
  \begin{align*}
  &\begin{array}{c|c|c|c}
    \begin{array}{c}
      \text{Starting} \\ \text{Point}
    \end{array} & \text{Direction} & \text{Limit in } \mb{TP}^3 &
    \begin{array}{c}
      \text{Limit in} \\
      (\mb{TP}^1)^3
    \end{array}
      \\ \hline
    c_e; (e \notin T) & (0,1,0) & [-\infty: F_2(c_{e}): - \infty : -\infty ] & (F_1 , \infty , F_3) \\
    a_e; e\notin T & (0,-1,-1) & [F_1(a_e) : -\infty : -\infty : 0 ] &  (F_1(a_e) , -\infty , -\infty ) \\
    a_e; e\in T & (0, 0, -1) & [F_1(a_e) : F_2(a_e) : -\infty : 0]  & \! (F_1(a_e) , F_2(a_e) , -\infty )\!\! \\
    p_e & (1,s(e),0) & [ -\infty : F_2(p_e) : -\infty : -\infty] & (\infty , \infty , F_3(p_e) ) \\ \hline
    q_e & (1,s(e),0) & [-\infty : F_2(q_e) : -\infty : -\infty]  & (\infty , \infty , F_3(q_e) )\\
    b_e; e\in T & (0, 0, 1) & [-\infty : -\infty : F_3(b_e) : -\infty ]  & (F_1(b_e), F_2(b_e), \infty)  \\
    b_e; e\notin T & (0,-1,1) & [-\infty : -\infty : F_3(b_e) : -\infty ] & (F_1(b_e), -\infty, \infty)  \\
    d_e; (e \notin T) & (0,1,0) & [-\infty : F_2(d_e) : -\infty : -\infty] & (F_1(d_e), \infty, F_3(d_e)) \\ \hline
    v \in V(G) & * & [-\infty: -\infty: F_3(v) : 0 ] & (-\infty , -\infty , F_3(v))
  \end{array} \\[1em]
  &* = \Big( -\deg(v), -\sum\limits_{e \ni v} s(e), 0 \Big)
  \end{align*}
  \caption{Directions of infinite rays and their limit in $\mb{TP}^3$ and $(\mb{TP}^1)^3$.}
  \label{tab:rays}
\end{table}

For each of the points $a_e, b_e, c_e, d_e, p_e, q_e$ as well as each vertex of $G$,
we have an infinite ray in $\Sigma$. For example the ray from $a_e$ to $a_e'$.
Let us refer to each of these rays as $p$-rays, $c$-rays, $a$-rays, etc.

In this section, we prove that image of the $a, b, c, d, p$, and $q$ rays do not intersect each other in $\R^3$,
or the image of the finite skeleton, $\Gamma$.
The intersections of these rays at the boundary strata of $\mb{TP}^3$ and $(\mb{TP}^1)^3$
is recorded in Table~\ref{tab:rays}.

The direction of each of these rays in the image $F(\Sigma)$ is given by looking at the sum of the incoming slopes at the point in $F$.
For reference, these directions are also recorded in Table~\ref{tab:rays}.

\begin{lemma}
  The image of $[c_e,c_e')$ or $[d_e,d_e')$  under $F$ intersects the image of $\Gamma$ only at $c_e$ or $d_e$ respectively.
\end{lemma}
\begin{proof}
The first two coordinates of the ray at $c_e$ and the ray at $d_e$ are identical, so we will only make a distinction between $c$-ray or $d$-ray when we start talking about the third coordinate.

A point on $F([c_e,c_e'))$ or $F([d_e,d_e'))$ is of the form
\[ F(c_e \text{ or } d_e) + \lambda (0,1,0) \]
for some $\lambda \ge 0$. Suppose that some point of this ray coincides with $F(x)$ for some $x \in \Gamma$,
belonging to an edge $e'$, which would mean $F(x) = F(c_e \text{ or } d_e) + (0,\lambda,0)$.

First, if $e' \in T$, then by \ref{S:c}, $F_2(x) < F_2(c_e) \le F_2(c_e) + \lambda$.
Therefore, we must have $e' \notin T$.

Let $v$ denote the vertex closest to $x$.
Then we have \[ \dist_{e'}(v, x) = F_1(x) = F_1(c_e) = \dist_{e}(v(e), c_e). \]
By the interval condition, this implies that $x \in [v, a_{e'}]$ or $x \in [b_{e'}, w]$.

Now, looking at the third coordinates, we have
\[ r(v) = F_3(x) = F_3(c_{e} \text{ or } d_e) = r(v(e) \text{ or } w(e)). \]
By \ref{R:distinct} we must have $v = v(e)$ or $v = w(e)$.
Since the edges outside $T$ do not share a vertex, this means $e = e'$.

Since $e = e'$ and $F_1(x) = F_1(c_{e})$, we either have $x = c_{e}$ or $x = d_{e}$. If we started with a $c$-ray, then $F_3(x) = F_3(c_e)$ implies $x = c_e$ because $F_3$ is antisymmetric on $[c_e,d_e]$ and likewise if we started with a $d$-ray.
\end{proof}

\begin{lemma} \label{lem:ae-ray}
  For $e \notin T$, the image of $[a_{e},a_{e}')$ and of $[b_{e},b_{e}')$ intersects the image of $\Gamma$ only at $a_{e}$ or $b_e$ respectively.
\end{lemma}
\begin{proof}
As before, the first two coordinates of the $a_e$ and $b_e$-rays are identical, so we will only make a distinction between $a$-ray or $b$-ray for the third coordinate.

Suppose that $x \in \Gamma$ and $F(x) = F(a_e \text{ or } b_e) + \lambda(0,-1,\pm 1)$.
Let $e'$ be an edge containing $x$.
Since $F_1(x) = F_1(a_{e})$, we have $x \in [c_{e'}, p_{e'}]$ or $x \in [q_{e'}, d_{e'}]$ by the interval condition.
Therefore, $F_2(x) \in [F_2(c_{e'}), F_2(p_{e'})]$.

On the other hand, by \ref{S:c} or \ref{S:disjoint} the distance between $F_2(x)$ and $F_2(a_{e})$ is quite large if $e' \ne e$.
Specifically, if $e' \ne e$ we have
\[
\lambda = F_2(a_{e}) - F_2(x) > \max F_3 - \min F_3 \ge \abs{F_3(a_e \text{ or } b_e) - F_3(x)} = \lambda.
\]
See Figure~\ref{fig:ae-ray-gap} for a picture of the situation.

Since this is impossible, we must have $e' = e$.
Now, from $F_1(x) = F_1(a_{e})$ we have either $x = a_{e}$ or $x = b_{e}$, and then we can use $F_3$ to distinguish between $a_e$ and $b_e$.
\end{proof}

\begin{lemma}
  For $e \in T$, the image of $[a_e,a_e')$ or $[b_e, b_e')$ intersects the image of $\Gamma$ only at $a_e$ or $b_e$ respectively.
\end{lemma}
\begin{proof}
Suppose that $x \in \Gamma$ and $F(x) = F(a_e \text{ or } b_e) + (0,0, \pm \lambda)$ for some $\lambda \in \mb R_{\ge 0}$.
Then in particular, $F_1(x) = F_1(a_e)$ and $F_2(x) = F_2(a_e)$ so by Lemma~\ref{lem:F1-F2} we have $x = a_e$ or $x = b_e$.

For the $[a_e,a_e')$-ray, we have $F_3(b_e) > F_3(a_e) \ge F_3(a_e) - \lambda = F_3(x)$. So we can't have $x = b_e$, hence we must have $x = a_e$.

Likewise, for the $[b_e,b_e')$-ray, we have $F_3(a_e) < F_3(b_e) \le F_3(b_e) + \lambda = F_3(x)$.
\end{proof}

\begin{lemma}
  The image of $[p_e, p_e')$ or $[q_e,q_e')$ intersects the image of $\Gamma$ only at $p_e$ or $q_e$, respectively.
\end{lemma}
\begin{proof}
Let $x \in \Gamma$ with $F(x) = F(p_e \text{ or } q_e) + \lambda(1,s(e),0)$ and $\lambda \ge 0$.
Let $e'$ be an edge that contains $x$ and $e \neq e'$.

Suppose, for now, that $x$ is closest to $v(e')$ since this part of the argument is symmetrical.

First, suppose $e, e' \in T$.
Then $F_1(x) = F_1(p_e) + \lambda$ means $F_2(x) = s(e')F_1(x) = s(e')F_1(p_e) + s(e')\lambda$.
But, on the other hand, $F_2(x) = F_2(p_e) + s(e)\lambda = s(e)F_1(p_e) + s(e)\lambda$.
This is impossible unless $e = e'$.

Next, because $\min\{\vf(x),\vf(p_{e'})\} = F_1(x) \ge F_1(p_e) = \vf(p_e)$, we have
$\vf(p_e) \le \vf(x) \le \vf(q_e)$ by the interval condition.
Thus, \[ \lambda = F_1(x) - F_1(p_e) \le \dist(p_e,q_e) \le \dist(a_e,b_e) \le \max F_3 - \min F_3. \]
We should think of $\lambda$ as being small.

If $e \notin T$ then already $F_2(p_e) + s(e)\lambda \ge F_2(p_e) > F_2(x)$ for any $x \in e \notin T$.

If $e \in T$ but $e' \notin T$ then we appeal to \ref{S:F2lambda} to see that this is impossible.

Thus, $e = e'$ and now things are no longer symmetric.
Now, since $F_1(p_e) = \max_{y \in e} F_1(y)$, it must be that $\lambda = 0$ and $x \in [p_e,q_e]$.
Since $F_3$ is injective on this interval, we have $x = p_e$ or $x = q_e$ depending on whether we started with a $p$-ray or a $q$-ray.
\end{proof}

\begin{figure}[htbp]
  \centering
  \begin{tikzpicture}
    %graph of F2|ei
    \draw (1,0) -- (2,2) -- (3,3) -- (4,5) -- (5,5) -- (6,3) -- (7,2)  -- (8,0);

    %graph of F2|e
    \draw (1,0) -- (2.2,1.2) -- (7.1,1.2) -- (8.3,0);

    %dashed lines through c_e and p_e
    \draw[dashed] (1,3) -- (8,3) node [right] {$F_2(a_{e})$};
    \draw[dashed] (1,2) -- (8,2) node [right] {$F_2(c_{e})$};
    \draw[dashed] (1,1) -- (8,1) node [right] {$F_2(x)$};

    \filldraw (3,3) circle (0.5mm);
    \filldraw (2,1) circle (0.5mm);

    \draw[<->] (3,1.2) -- (3,2);
        \node[right] at (3,1.6) {$> \max F_3 - \min F_3$};

    \draw[<->] (1.7,1) -- (1.7,3);
        \node[above left] at (1.7,2) {$\lambda=$};
    \end{tikzpicture}
  \caption{Situation in Lemma~\ref{lem:ae-ray}}
  \label{fig:ae-ray-gap}
\end{figure}
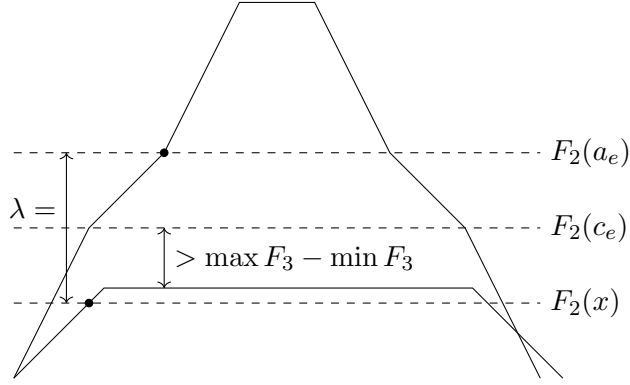

\subsubsection{Comparing between rays}
\begin{note}
These proofs are all quite short and just come down to requiring some parameters being distinct.
\end{note}

\begin{lemma}
Any pair of distinct $c$-rays or pair of distinct $d$-rays do not intersect.
\end{lemma}
\begin{proof}
An intersection between two $c$-rays has the form $F(c_{e}) + (0,\lambda,0) = F(c_{e'}) + (0,\mu,0)$ for some $\lambda$ and $\mu$.
Because we chose distinct values for $F_1(c_{e}) = \dist_{e}(c_{e}, v(e))$,
and $F_1(c_{e}) = F_1(c_{e'})$, therefore $e = e'$.

For $d$-rays, simply change $c$ to $d$ and $v(e)$ to $w(e)$.
\end{proof}

\begin{lemma}
  Any pair of distinct $p$-rays or pair of distinct $q$-rays do not intersect.
\end{lemma}
\begin{proof}
Two $p$-rays look like $F(p_e) + (\lambda,s(e)\lambda,0) = F(p_{e'}) + (\mu,s(e)\mu,0)$.
Because we chose distinct values of $F_3(p_e) = r(v(e)) + \dist_e(p_e, a_e)$,
and $F_3(p_e) = F_3(p_{e'})$, therefore $e = e'$.

Likewise, we chose distinct values for $F_3(q_e)$ so no pair of distinct $q$-rays can intersect.
\end{proof}

\begin{lemma}
Any pair of distinct $a$-rays or $b$-rays do not intersect.
\end{lemma}
\begin{proof}
The first coordinate of every point in an $a$-ray or $b$-ray is $F_1(a_e)$. By \ref{R:A-distinct}, these quantities are distinct.
\end{proof}

\begin{lemma}
No pair of $a, b, c, d, p$, or $q$-rays intersect, except possibly $a$ with $b$, $c$ with $d$ and $p$ with $q$.
\end{lemma}
\begin{proof}
Note that the first coordinates of these rays are $F_1(a_e), F_1(c_e)$ and $F_1(p_e) + \lambda$ respectively.
By the interval condition, these are ordered
\[
 F_1(a_e) < F_1(c_e) < F_1(p_e) \le F_1(p_e) + \lambda. \qedhere
\]
\end{proof}

\begin{lemma}
  An $a$-ray cannot intersect a $b$-ray.
\end{lemma}
\begin{proof}
  Because the values of $F_1(a_e) = F_1(b_e)$ are distinct, an $a$-ray can only possibly intersect the $b$-ray belonging to the same edge. But then
  \[ F_3(a_e) - \lambda \le F_3(a_e) < F_3(b_e) \le F_3(b_e) + \mu \]
  for all $\lambda, \mu \ge 0$.
\end{proof}

\begin{lemma}
  A $c$-ray cannot intersect a $d$-ray.
\end{lemma}
\begin{proof}
  Because the values of $F_1(c_e) = F_1(d_e)$ are distinct, a $c$-ray can only possibly intersect the $d$-ray belonging to the same edge. But then $F_3(c_e) < F_3(d_e)$.
\end{proof}

\begin{lemma}
  A $p$-ray cannot intersect a $q$-ray.
\end{lemma}
\begin{proof}
  Because the values of $F_1(p_e) = F_1(q_e)$ are distinct, a $p$-ray can only possibly intersect the $q$-ray belonging to the same edge. But then $F_3(p_e) < F_3(q_e)$.
\end{proof}

\begin{lemma}
Two distinct vertex rays do not intersect.
\end{lemma}
\begin{proof}
Note that the third coordinate of a vertex ray is $F_3(v) = r(v)$ and these values are distinct by \ref{R:distinct}.
\end{proof}

\begin{lemma} \label{vertex c a p}
A vertex ray does not intersect an $c, d, a, b, p$, or $q$-ray.
\end{lemma}
\begin{proof}
Note that the first coordinate of a vertex ray is
\[ F_1(v) - \lambda \deg(v) = - \lambda \deg(v) \le 0 < F_1(c_e) < F_1(a_e) < F_1(p_e). \qedhere
\]
\end{proof}

\section{Fully and totally faithfulness} \label{sec:compactifications}

In this section we prove Theorem \ref{ThmA} from the introduction.
The majority of the work was done in the previous section.
In this section we show that all the assumptions we made there
can actually be achieved.
We fix a Mumford curve $X$.

\begin{theorem} \label{thm totally faithful}
Let $Y$ be a proper toric variety of dimension three.
Then there exists a morphism $\varphi \colon X \to Y$
such that the induced tropicalization is totally faithful.
\end{theorem}

\begin{proof}
Let $\Gamma$ be a finite skeleton of $X$.
By simply adding a leaf edge to $\Gamma$, we may assume that $\Gamma$ has a leaf edge.
We pick a graph model $G = (V, E)$ for the $\Lambda$-metric graph $\Gamma$,
and we chose the points $c_e,a_e, p_e, q_e, b_e, d_e$
satisfying the interval condition, and we pick values $r(v)$
such that \ref{R:distinct} and \ref{R:A-distinct} are satisfied.
Now since we assumed that $\Gamma$ has a leaf edge, Lemma \ref{Sispossible}
shows that we can pick $s(e)$ for $e \in E$ such that
 \ref{S:distinct}--\ref{S:coprime} are satisfied.

The rational functions $f_1,f_2,f_3$ constructed in Propositions \ref{D3 lift}, \ref{D1 lift} and \ref{D2 lift}
define a rational map $X \to \G_m^3$.
Identifying the dense torus of $Y$ with $\G_m^3$ and using the fact that both $X$ and $Y$ are proper,
we obtain a morphism $\varphi \colon X \to Y$.

By Proposition \ref{prop choosing parameters}, the map $\trop_{\varphi}\!|_{\mathring \Sigma(\varphi)}$ is injective.
By Proposition \ref{prop inj implies faithful}, this means that $\trop_{\varphi}$ is totally faithful.
\end{proof}

\begin{lemma} \label{Sispossible}
If $\Gamma$ has a leaf edge, it is possible to pick $s(e)$ in a way such that
they satisfy \ref{S:distinct}--\ref{S:coprime}.
\end{lemma}
\begin{proof}
Let us focus on \ref{S:coprime} first.
Pick any set of numbers $s(e)$ for all $e \in E$.
We pick a point $z$ that lies in the interior of an edge and subdivide $\Gamma$ by introducing $z$ as a vertex.
Let $v$ and $w$ be two vertices of $\Gamma$, joint by an edge $e$.
Note that one can always achieve that \ref{S:coprime} holds at $v$ by changing $s(e)$ an appropriate amount.

Note further that for any vertex $v$ except $z$, their exists a vertex $w$ that lies closer to $z$ that $v$.
For every $v$ fix such a choice $w_v$.
Now working ones way closer to $z$, by each time changing $s(e_v)$,
where $e_v$ is the edge joining $v$ and $w_v$,
we get $S(6)$ to hold for all vertices except $z$.
We now add a leaf edge $e$ at $z$ and are done,
since we can pick $s(e)$ in a way such that \ref{S:coprime} holds at $z$.

The other properties can all be achieved by making the $s(e)$ very large
with large differences between them.
This can be achieved by adding multiples of $\prod_{v \in \Gamma} \deg(v)$ to the $s(e)$,
so they remain coprime.
\end{proof}

Now let us take a closer look at two particular toric varieties: $\Proj^3$ and $(\Proj^1)^3$.
The functions $f_1, f_2, f_3$ are the ones constructed in Propositions \ref{D3 lift}, \ref{D1 lift} and \ref{D2 lift}
with the parameters chosen as in Section \ref{sec:params}.

\begin{proposition}
Let $\varphi \colon X \to \Proj^3; \; \; x \mapsto [f_1(x) : f_2(x) : f_3(x) : 1]$.
Then the induced tropicalization is \emph{not} fully faithful.
\end{proposition}

\begin{theorem} \label{thm fully faithful}
Let $\varphi \colon X \to (\Proj^1)^3; \; \; x \mapsto (f_1(x), f_2(x), f_3(x))$.
Then the induced tropicalization is fully faithful.
\end{theorem}

\begin{proof}
Both these statements follow from Table~\ref{tab:rays} that lists the endpoints
of the rays in the respective compactifications together with the requirements of Section~\ref{sec:distinctness} that force the endpoints to be distinct.
\end{proof}

\section{Resolution of singularities}\label{sec:res-of-singularities}

\subsection{A conceptual approach}

Throughout this section, we fix a Mumford curve $X$
and a morphism $\varphi \colon X \to Y$ for a toric variety $Y$
that induces a fully faithful tropicalization.

\begin{definition}\label{def:local-degree-of-non-smoothness}
  Let $\Trop_{\vf}(X)$ be the corresponding tropical curve in $\R^n$ and let $x \in \Trop_\vf(X)$.
  We define the \emph{local degree of non-smoothness} of $\Trop_\vf(X)$ at $x$ to be
  \begin{align}
  \label{eq local degree of smoothness}
  n_{\vf}(x) = \deg(x) - 1 - \max \{ k \mid &\text{ tangent vectors } v_1,\dots,v_k \\
  \nonumber 				                        &\text{ span a saturated lattice of rank } k\}.
  \end{align}
\end{definition}

\begin{note}
Consider the tropical curve in Figure \ref{fig:wagner-construction}.
The circled point $x$ has degree $4$, one can find two tangent vectors that span $\Z^2$,
but any three will still span $\Z^2$.
We conclude that $n_\varphi(x) = 1$.

In general,  $x$ is a smooth point if and only if $n_{\vf}(x) = 0$.
\end{note}

\begin{theorem}\label{thm:resolution}
With notation as above, there exists a rational function $f$ on $X$ such that if we denote by
$\varphi' \colon X \to Y \times \Proj^1, x \mapsto (\varphi(x), f(x))$
the associated embedding, $\varphi'$ is fully faithful and
\begin{align*}
n_{\varphi'}(z) =
\begin{cases} n_{\varphi}(z) -1 &\text{ if } n_{\varphi}(z) > 0 \\
0 &\text{ if } n_{\varphi}(z) = 0
\end{cases}
\end{align*}
for all $z \in \Sigma(\varphi)$.
\end{theorem}

\begin{figure}
  \begin{tikzpicture}
    \draw  (0,0) -- (1,1) -- (3,1) -- (4,0) -- (5,0);
    \draw  (0,0) -- (-1,-1) -- (-3,-1) -- (-4,0) -- (-5,0);
    \draw[dashed] (-4,0) -- (4,0);

    \draw (-5,-1.3) -- (5,-1.3);
    \node[below] at (-2,-1.3) {$e_0$};
    \node[below] at (2,-1.3) {$e_k$};
    \node[below] at (0,-1.3) {$v$}; \fill (0,-1.3) circle (2pt);
    \node[below] at (1,-1.3) {$p(v)_{k-2}$}; \fill (1,-1.3) circle (2pt);
    \node[below] at (3,-1.3) {$q(v)_{k-2}$}; \fill (3,-1.3) circle (2pt);
    \node[below] at (4.4,-1.3) {$r(v)_{k-2}$}; \fill (4,-1.3) circle (2pt);
    \node[below] at (-1,-1.3) {$p(v)_{0}$}; \fill (-1,-1.3) circle (2pt);
    \node[below] at (-3,-1.3) {$q(v)_{0}$};  \fill (-3,-1.3) circle (2pt);
    \node[below] at (-4,-1.3) {$r(v)_{0}$};  \fill (-4,-1.3) circle (2pt);
  \end{tikzpicture}

  \caption{The graph of $F_{e_k}(v)$ along the edges $e_k(v)$ and $e_0(v)$.
  The function $F_{e_k}$ is constant $0$ on all other edges.}
  \label{figure resolution}
\end{figure}
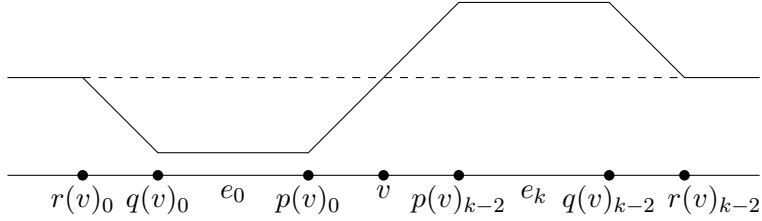

\begin{proof}
For each vertex $z$ in $\Sigma_{\varphi}$ such that $n_{\varphi'}(z) > 0$,
pick tangent vectors $e(z)_2,\dots,e(z)_{k+1}$ which span a saturated lattice as in\eqref{eq local degree of smoothness}.
Further, fix two other adjacent edges $e(z)_0$ and $e(z)_1$.
In both $e(z)_0$ and $e(z)_1$ we choose points $p(z)_i, q(z)_i, r(z)_i\in e(v)_i$
that are close to $z$, in the sense that they are closer to $z$ then to the other vertex of $e(z)_i$.
Further they should satisfy $\dist(p(z)_i, z) = \dist(q(z)_i, r(z)_i)$.

We now let
\begin{align*}
D_z &= -p(z)_{1} - q(z)_{1} + r(z)_{1} + p(z)_{0} + q(z)_0 - r(z)_0 \text{ and } \\
D &= \sum_{z \in \Sigma, n(z) > 1 } D_z.
\end{align*}
Let $\Gamma$ be the finite skeleton obtained from $\Sigma(\varphi)$ that is obtained
by removing the infinite edges.
Let $\Gamma'$ be a subdivision of $\Gamma$ such that all the $r(v), q(v), p(v)$ are vertices.
Now we pick edges $e_1,\dots,e_g$ of $\Gamma$ that form the complement of a spanning tree
and in each edge we pick points $s_1^j,s_2^j, s_3^j, s_4^j$ that occur on $e_j$ in this order and satisfy
$d_{e_j}(s_1^j, s_2^j) = d_{e_j}(s_3^j, s_4^j)$.
Denote by $P$ the divisor
\begin{align*}
P = \sum_{j=1}^g s_1^j - s_2^j - s_3^j + s_4^j
\end{align*}
on $\Gamma$.
Now by the lifting theorem (Theorem~\ref{lifting thm convenient}),
we find lifts of all points in the support of $D + P$
such that the divisors $D'$ and $P'$ satisfy that $D' + P'$ is principal
and $\tau_* P' = P$ and $\tau_* D' = D$.

Let $f$ be such that $\divisor(f) = D' + P'$.
We claim that $f$ has the required properties.
One checks easily that the tropicalization is again fully faithful.

Let $z$ be a vertex of $\Sigma(\varphi)$ and $v_1,\dots,v_{k+1}$
be as above.
Then the images of the tangent vectors at $z$ are now
\begin{equation} \label{vectors new lattice}
(v_1, 1) \; (v_2, 0) \dots (v_{k+1}, 0).
\end{equation}
The lattice $L'$ spanned by the vectors in \eqref{vectors new lattice} is of rank $k+1$.
We have $\Z^{n+1} / L' \cong \Z^n / L$, using the map
\begin{align*}
\Z^{n+1} \to \Z^n; (x_1,\dots,x_n) \mapsto (x_1 - v^1 x_{n+1}, \dots , x_n - v^n x_{n+1}),
\end{align*}
where $v_1 = (v^1,\dots,v^n)$.
In particular, $L'$ is saturated.
Since we do not add any edges at $z$, we have $n_{\varphi'}(z) = n_{\varphi}(z) -1$.

If $z$ is a vertex with $n_{\varphi}(z) = 1$, then $\log \vert f \vert$ is constant in a neighborhood of
$z$ and thus $n_{\varphi'}(z) = 1$.

If $z$ is one of the points in the support of $D$, then it is contained in an edge of $\Sigma$.
Denote by $w$ the vector in the direction of $e$ in $\Trop_{\varphi}(X)$.
Then $z$ is of degree $3$ in $\Sigma'$ and the set of direction vectors is either
\begin{align*}
\{(w, 1); (w, 0); (0, -1)\} \text{ or } \{(w, -1) ; (w, 0) ; (0, 1) \}.
\end{align*}
In particular, those span a saturated lattice of rank $2$ and $n_{\varphi'}(z) = 1$.
\end{proof}

\begin{corollary} \label{Cor resolution of singularities}
Let $n(\varphi) = \max_{z \in \Sigma} (n_{\varphi}(x))$.
Then there exist $n(\varphi)$ rational functions $f_1,\dots,f_{n(\varphi)}$ on $X$ such that if we denote by
\begin{align*}
\varphi' \colon X &\to Y \times (\Proj^1)^{n(\vf)}, \\
x &\mapsto ( \varphi(x), f_1(x), \dots, f_{n(\varphi)} )
\end{align*}
the associated embedding, $\Trop_\varphi'(X)$ is smooth.
\end{corollary}

\begin{proof}
This follows by applying Theorem \ref{thm:resolution} inductively until $n_{\varphi'}(z) = 0$ for all $z$.
\end{proof}

\subsection{Application to our situation}

In this section, we prove the following theorem:
\begin{theorem} \label{theorem smooth embedding}
Let $X$ be a Mumford curve.
Let $C$ be the maximal degree of a vertex on the minimal skeleton $\Gamma$ of $X$.
Then there exists a map $X \to (\Proj^1)^{C+2}$
that induces a smooth tropicalization of $X$.
\end{theorem}

\begin{note}
 This is $3$ more than the optimal bound of $C-1$ that is determined
by the definition of smoothness in terms of spans of direction vectors (c.f. \S\ref{sec:ff-and-smooth}).
\end{note}

\begin{proof}
Let $X \to (\Proj^1)^{3}$ be a map that induces by fully faithful tropicalization,
as in Theorem~\ref{thm fully faithful}.
Note that the maximum degree of a vertex in $\Sigma(\varphi)$ is $C+1$, as
we add one infinite edge at every vertex.
Let $z$ be a vertex of $\Gamma$ and $e_1,\dots,e_k$ the adjacent edges.
Let $e_0$ be the adjacent infinite edge in $\Sigma$.
The tangent vectors in the tropicalization we constructed are of the form
\[
(1, s_{e_1}, 0), \dots , (1, s_{e_k}, 0), (k, - \sum s_{e_i}, 0).
\]
Unfortunately, no two of these span a saturated lattice of rank $2$.
We conclude that $n_{\varphi}(z) = \deg_{\Sigma}(x) - 2 = \deg_{\Gamma}(z) -1$.

Since all other $z \in \Sigma(\varphi)$ are at most trivalent, we conclude that
$n(\varphi) = C - 1$.

The result now follows from Corollary~\ref{Cor resolution of singularities}
and the fact that $C -1 + 3 = C + 2$.
\end{proof}

\begin{corollary}\label{cor:smooth-tropicalization}
Let $X$ be a Mumford curve of genus $g$.
Then there exists a map $X \to (\Proj^1)^{2g+2}$
that induces a smooth tropicalization of $X$.
\end{corollary}
\begin{proof}
The minimal skeleton of a genus $g$ Mumford curve has first Betti number $g$.
Any vertex in a graph with genus $g$ has degree at most $2g$.
Thus the Corollary follows from Theorem \ref{theorem smooth embedding}.
\end{proof}

\section{A genus 2 curve} \label{sec:genus2}
A construction for tropicalizing certain genus 2 Mumford curves has been given by Wagner \cite{Wag}.
For skeleta consisting of two loops joined at a common point,
his construction is pictured in Figure~\ref{fig:wagner-construction}.
In ambient dimension $2$, there is an intersection point.
Wagner fixes this by adding in a third rational function to resolve the crossing in ambient dimension $3$.

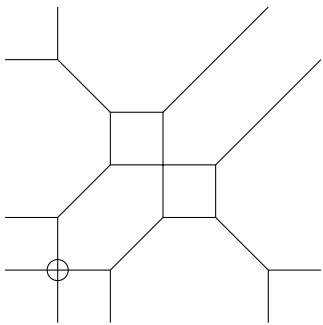
\begin{figure}[htbp]
  \centering
  \begin{tikzpicture}[scale=0.7]
    \draw (0,0) -- (0,1) -- (-1,1) -- (-1,0) -- (1,0) -- (1,-1) -- (0,-1) -- (0,0);
    \draw (0,1) -- (2,3);
    \draw (1,0) -- (3,2);
    \draw (-1,0) -- (-2,-1);
      \draw (-2,-1) -- (-3,-1);
      \draw (-2,-1) -- (-2,-3);
    \draw (0,-1) -- (-1,-2);
      \draw (-1,-2) -- (-3,-2);
      \draw (-1,-2) -- (-1,-3);
    \draw (-1,1) -- (-2,2);
      \draw (-2,2) -- (-3,2);
      \draw (-2,2) -- (-2,3);
    \draw (1,-1) -- (2,-2);
      \draw (2,-2) -- (3,-2);
      \draw (2,-2) -- (2,-3);
    \draw (-2,-2) circle (0.2);
  \end{tikzpicture}
  \caption{First step of Wagner's construction of a tropicalization of a genus two curve with an intersection circled.}
  \label{fig:wagner-construction}
\end{figure}

Wagner's construction does not consider the singularity at the four-valent vertex and
further analysis is required to show this point can be made smooth.

In this section, we show how to approach such tropicalization questions combinatorially from a rough-draft picture and
how resolving this four-valent point comes ``for free'' with our approach.

\subsection{Picturing the construction}
Picturing how the skeleton should be embedded in $\mb{TP}^{3}$ tells us how to construct the divisors.
 The first picture we visualize is just two hexagons attached at a common vertex and
contained in the planes $z = 0$ and $x = y$ respectively.
Second, we figure out how all the infinite rays should go so that the rays have directions
$(-1,0,0)$ or $(0,-1,0)$ or $(0,0,-1)$ or $(1,1,1)$
so that we can guarantee that they do not intersect in the boundary strata of $\mb{TP}^3$.
This gives us the picture of Figure~\ref{fig:genus2roughDraft}.

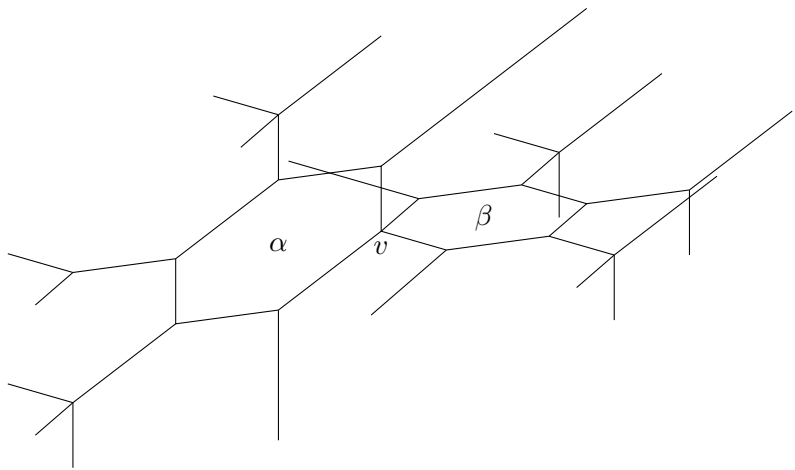
\begin{figure}[htbp]
  \centering
  \tdplotsetmaincoords{60}{30}
  \begin{tikzpicture}[tdplot_main_coords]
    % alpha circle
    \draw (0,0,0) -- (0,0,1) -- (-1,-1,1) -- (-2,-2,0) -- (-2,-2,-1) -- (-1,-1,-1) -- (0,0,0);
    % beta circle
    \draw (0,0,0) -- (1,0,0) -- (2,1,0) -- (2,2,0) -- (1,2,0) -- (0,1,0) -- (0,0,0);
    % alpha rays
    \draw (0,0,1) -- (2,2,3);
    \draw (-1,-1,1) -- (-1,-1,2);
      \draw (-1,-1,2) -- (0,0,3);
      \draw (-1,-1,2) -- (-2,-1,2);
      \draw (-1,-1,2) -- (-1,-2,2);
    \draw (-2,-2,0) -- (-3,-3,0);
      \draw (-3,-3,0) -- (-3,-4,0);
      \draw (-3,-3,0) -- (-4,-3,0);
    \draw (-2,-2,-1) -- (-3,-3,-2);
      \draw (-3,-3,-2) -- (-3,-3,-3);
      \draw (-3,-3,-2) -- (-3,-4,-2);
      \draw (-3,-3,-2) -- (-4,-3,-2);
    \draw (-1,-1,-1) -- (-1,-1,-3);
    % beta rays
    \draw (1,0,0) -- (1,-2,0);
    \draw (2,1,0) -- (3,1,0);
      \draw (3,1,0) -- (4,2,1);
      \draw (3,1,0) -- (3,1,-1);
      \draw (3,1,0) -- (3,0,0);
    \draw (2,2,0) -- (3,3,0);
      \draw (3,3,0) -- (4,4,1);
      \draw (3,3,0) -- (3,3,-1);
    \draw (1,2,0) -- (1,3,0);
      \draw (1,3,0) -- (2,4,1);
      \draw (1,3,0) -- (1,3,-1);
      \draw (1,3,0) -- (0,3,0);
    \draw (0,1,0) -- (-2,1,0);
    % labels
    \draw (0,0,0) node [below] {$v$};
    \draw (-1,-1,0) node {$\alpha$};
    \draw (1,1,0) node {$\beta$};
  \end{tikzpicture}
  \caption{First draft of how the genus 2 skeleton is embedded in $\mb{TP}^3$.}
  \label{fig:genus2roughDraft}
\end{figure}

%\subsubsection{First draft construction}
Let $\Xan$ be the analytification of a curve whose skeleton consists of two loops,
$\alpha$ and $\beta$, connected at a common point, $\omega$.

In order to form the hexagons, we need to choose $5$ points spaced equidistant around each loop of the skeleton.
To that end, let  $\alpha_1,\dots,\alpha_5$ be points spaced equidistant around
$\alpha$ and $\beta_1,\dots,\beta_5$ equidistant around $\beta$. See Figure~\ref{fig:genus2skeleton}.

\begin{figure}[tbp]
  \centering
  \begin{tikzpicture}[scale=0.7]
    \draw (2.5,0) circle (2.5);
    \draw (-2,0) circle (2);
    \filldraw (0,0) circle (0.5mm);

    \draw [right] (0,0) node {$\omega$};

    \foreach \x in {1,...,5} {
      % dots at each point
      \filldraw[black] (-2,0) + ({60*\x}:2) circle(0.5mm);
      \filldraw[black] (2.5,0) + ({180 - 60*\x}:2.5) circle(0.5mm);
      % labels
      \draw (-2,0) + ({60*\x}:1.6) node {$\alpha_{\x}$};
      \draw[black] (2.5,0) + ({180 - 60*\x}:2.1) node {$\beta_{\x}$};
    }

    \foreach \x in {1,...,4} {
      \filldraw[black] (-2,0) + ({12*\x}:2) circle (0.5mm);
      \filldraw[black] (2.5,0) + ({120 - 12*\x}:2.5) circle (0.5mm);
      \draw (-2,0) + ({12*\x}:2.4) node {$\gamma_{\x}$};
      \draw (2.5,0) + ({120 - 12*\x}:2.9) node {$\delta_{\x}$};
    }
  \end{tikzpicture}
  \caption{Skeleton $\Gamma$ of $X$.}
  \label{fig:genus2skeleton}
\end{figure}
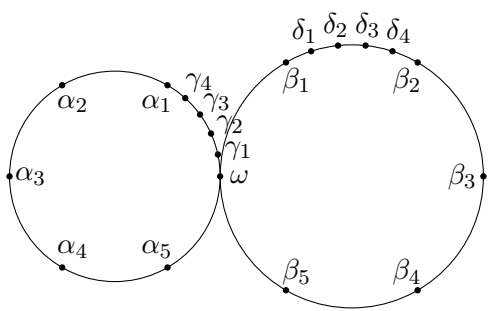

We will arrange so that $\alpha$ is the hexagon in the $x = y$ plane and $\beta$ is in the $z = 0$ plane.

For the first divisor, we note that the $x$-coordinate stays constant between $\omega$ and $\alpha_1$,
then decreases linearly, with slope $1$, from $\alpha_1$ to $\alpha_3$ and so on.
Writing down where the slope changes gives us the divisor
\[ \alpha_1 - \alpha_3 - \alpha_4 + \beta_2 + \beta_3 - \beta_5. \]

Doing the same for the $y$-coordinate, gives us the divisor
\[ \alpha_1 - \alpha_3 - \alpha_4 - \beta_1 + \beta_3 + \beta_4. \]
The $\alpha$-hexagon is contained in the $x = y$ plane, so it makes sense that the first three
terms of each divisor are identical.
However, this presents a problem because we need the lifting theorem to choose lifts for us on a break-divisor and
we don't have any points we can allow the lifting theorem to choose for us on the $\alpha$-cycle.

We also need to consider the infinite rays.
For example, at $\alpha_2$ we have a ray going straight up (direction: $(0,1,0)$) and
then branching in the directions $(-1,0,0)$, $(0,-1,0)$ and $(1,1,1)$.
Thus we have two rays that have a non-zero $x$-coordinate and two rays that have a non-zero $y$-coordinate.
Therefore, we need to lift $\alpha_2$ to $x_{2,0} - x_{2,1}$ and $x_{2,0} - x_{2,2}$ for the $x$ and $y$ coordinates respectively.

\subsection{A proper construction}
In order to construct this embedding properly,
we first need to choose 4 points $\gamma_1,\dots,\gamma_4$ spaced equidistant
between two previously marked points, let's say $\omega$ and $\alpha_1$ and
another four points $\delta_1,\dots,\delta_4$ spaced equidistant between $\beta_1$ and $\beta_2$.
These points provide for us break-divisors which we can feed into Theorem~\ref{lifting thm convenient}.
These points are also pictured in Figure~\ref{fig:genus2skeleton}.

As in Section~\ref{sec:construction},
we apply Theorem~\ref{lifting thm convenient} to the data of Table~\ref{tab:listofdivisors}
where the break divisors are the sum of the circled quantities.
This yields three piecewise-linear function $F_1, F_2, F_3$ from the extended skeleton to $\TT \mb P^1$.

\begin{table}[htbp]
    \centering
    \hspace*{\fill}
    $\begin{array}[t]{c|c}
        \tau_* D_1 & D_1  \\ \hline
        + \alpha_1 & \circled{x_1} \\
                   & x_{2,0} - x_{2,1} \\
        - \alpha_3 & - x_{3,1} \\
        - \alpha_4 & - x_{4,1} \\ \hline
        +  \beta_2 & \circled{y_{2,0}} \\
        +  \beta_3 & y_{3,0} \\
                   & y_{4,0} - y_{4,1} \\
         - \beta_5 & - y_5 \\ \hline
                   & u_{2,0} - u_{2,1} \\
                   & u_{3,0} - u_{3,1} \\ \hline
                   & v_{2,0} - v_{2,1} \\
                   & v_{3,0} - v_{3,1}
    \end{array}$
    \hspace*{\fill}
    \vrule
    \hspace*{\fill}
    $\begin{array}[t]{c|c}
        \tau_* D_2 & D_2  \\ \hline
        + \alpha_1 & x_1 \\
                   & x_{2,0} - x_{2,2} \\
        - \alpha_3 & - x_{3,2} \\
        - \alpha_4 & - x_{4,2} \\ \hline
         - \beta_1 & - \circled{y_1} \\
                   & y_{2,0} - y_{2,2} \\
         + \beta_3 & y_{3,0} \\
         + \beta_4 & y_{4,0} \\ \hline
        - \gamma_1 & - \circled{u_1} \\
        + \gamma_2 &   u_{2,0} \\
        + \gamma_3 &   u_{3,0} \\
        - \gamma_4 & - u_4 \\ \hline
                   & v_{2,0} - v_{2,2} \\
                   & v_{3,0} - v_{3,2}
    \end{array}$
    \hspace*{\fill}
    \vrule
    \hspace*{\fill}
    $\begin{array}[t]{c|c}
        \tau_* D_3 & D_3  \\ \hline
        + \alpha_1 & x_1 \\
        + \alpha_2 & x_{2,0} \\
        - \alpha_4 & - x_{4,3} \\
        - \alpha_5 & -\circled{x_5} \\ \hline
                   & y_{2,0} - y_{2,3} \\
                   & y_{3,0} - y_{3,3} \\
                   & y_{4,0} - y_{4,3} \\ \hline
                   & u_{2,0} - u_{2,3} \\
                   & u_{3,0} - u_{3,3} \\ \hline
        - \delta_1 & - \circled{v_1}   \\
        + \delta_2 & v_{2,0} \\
        + \delta_3 & v_{3,0} \\
        - \delta_4 & - v_4
    \end{array}$
    \hspace*{\fill} \vspace{1 em}
    \caption{Divisors on $\Gamma$ and on $\Xan$. Lifts are chosen first for $D_1$, then $D_2$, then $D_3$.}
    \vspace{-1 em}
    \label{tab:listofdivisors}
\end{table}

Here the notation for the lifts is as follows:
\begin{itemize}
\item $x$'s correspond to $\alpha$'s, $y$'s to $\beta$'s, $u$'s to $\gamma$'s and $v$'s to $\delta$'s
\item lifts with a single subscript are the unique lift of that point in $\Xan$ (and this lift is consistent for each divisor)
\item for a lift with two subscripts, e.g.\ $x_{i,j}$,
the first subscript represents the index of the corresponding point of $\Gamma$ (so $x_{i,j}$ is a lift of $\alpha_i$).
The second subscript corresponds to which divisor the lift is for (e.g.\ $x_{i,j}$ is a lift for $D_j$).
If the second subscript is $0$, the lift appears in all three of $D_1,D_2,D_3$ (and again, the lift is consistent).
\end{itemize}

We choose multiple lifts of the same point of $\Gamma$ in order to ensure
the resulting tropicalization is ``injective at infinity'' i.e.\ we have an embedding in $\mb{TP}^3$.
This is achieved by choosing the lifts in such a way that all the infinite rays have directions
$(-1,0,0), (0,-1,0), (0,0,-1)$ or $(1,1,1)$.

Having done this, we need to ensure smoothness,
and this requires us to choose the lifts over a point $p$ to share a common initial segment of length
$\ell(p)$ as in Figure~\ref{fig:genus2commonedge}.
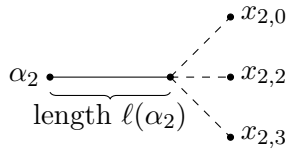
\begin{figure}[htbp]
  \centering
  \begin{tikzpicture}[scale=0.8]
      \foreach \x/\y in {0/0, 2/0, 3/1, 3/0, 3/-1}
          \filldraw (\x,\y) circle (0.5 mm);
      \draw [left] (0,0) node {$\alpha_2$};
      \draw (0,0) -- (2,0);
      \draw [decoration={brace, raise=2 mm, mirror}, decorate] (0,0) -- node [below=2 mm] {length $\ell(\alpha_2)$} (2,0);
      \foreach[count=\i, evaluate=\i as \y using int(2 - \i)] \j in {0,2,3} {
          \draw [right] (3, \y) node {$x_{2,\j}$};
          \draw [dashed] (2,0) -- (3,\y); }
  \end{tikzpicture}
  \caption{$\alpha_2$ and its lifts (dashed lines are infinite).}
  \label{fig:genus2commonedge}
\end{figure}

\begin{proposition}
The data in Table~\ref{tab:listofdivisors} allows us, via Theorem~\ref{lifting thm convenient},
to find rational functions $f_1,f_2,f_3$ on $\Xan$ whose divisors are $D_1, D_2,D_3$ and
such that $\divisor(\log |f_i|) = \tau_* D_i$ for all $i$. As before, we let $F = (F_1,F_2,F_3)$.

  For convenience, we will assume that $F(v) = (0,0,0)$. \hfill $\qed$
\end{proposition}

\subsection{Injective, smooth and fully-faithful} \label{sec:genus2isff}
The goal of this section is to explain why this construction is smooth and fully-faithful and
how to choose the appropriate parameters to make the construction injective.

First, we will explain how the picture we started with (Figure~\ref{fig:genus2roughDraft}) does
not have any crossings.
Then we will explain how to choose the data corresponding to
$\gamma_1,\dots,\gamma_4,\delta_1,\dots,\delta_4$ to get an injective lift.

\begin{figure}[tbp]
  \centering
  \tdplotsetmaincoords{90}{45}
  \begin{tikzpicture}[tdplot_main_coords, scale=0.6]
    % alpha circle
    \draw (0,0,0) -- (0,0,1) -- (-1,-1,1) -- (-2,-2,0) -- (-2,-2,-1) -- (-1,-1,-1) -- (0,0,0);
    % beta circle
    \draw [dashed] (0,0,0) -- (2,2,0);
    % alpha rays
    \draw (0,0,1) -- (2,2,3);
    \draw (-1,-1,1) -- (-1,-1,2);
      \draw (-1,-1,2) -- (0,0,3);
    \draw (-2,-2,0) -- (-3,-3,0);
    \draw (-2,-2,-1) -- (-3,-3,-2);
      \draw (-3,-3,-2) -- (-3,-3,-3);
    \draw (-1,-1,-1) -- (-1,-1,-3);
    % beta rays
    \draw (2,2,0) -- (3,3,0);
      \draw (3,3,0) -- (4,4,1);
      \draw (3,3,0) -- (3,3,-1);
    % labels
    \draw (0,0,0) node [below] {$v$};
    \draw (-1,-1,0) node {$\alpha$};
  \end{tikzpicture} \vrule
  \tdplotsetmaincoords{0}{0}
  \begin{tikzpicture}[tdplot_main_coords, scale=0.6]
    % alpha circle
    \draw [dashed] (0,0,0) -- (-2,-2,0);
    % beta circle
    \draw (0,0,0) -- (1,0,0) -- (2,1,0) -- (2,2,0) -- (1,2,0) -- (0,1,0) -- (0,0,0);
    % alpha rays
    \draw (-2,-2,0) -- (-3,-3,0);
      \draw (-3,-3,0) -- (-3,-4,0);
      \draw (-3,-3,0) -- (-4,-3,0);
    % beta rays
    \draw (1,0,0) -- (1,-2,0);
    \draw (2,1,0) -- (3,1,0);
      \draw (3,1,0) -- (3,0,0);
    \draw (2,2,0) -- (3,3,0);
    \draw (1,2,0) -- (1,3,0);
      \draw (1,3,0) -- (0,3,0);
    \draw (0,1,0) -- (-2,1,0);
    % labels
    \draw (0,0,0) node [below] {$v$};
    \draw (1,1,0) node {$\beta$};
  \end{tikzpicture}
  \caption{The $x = y$ and $z = 0$ planes in our construction.
Dashed lines represent where the other hexagons are (outside the planes).}
  \label{fig:genus2planes}
\end{figure}
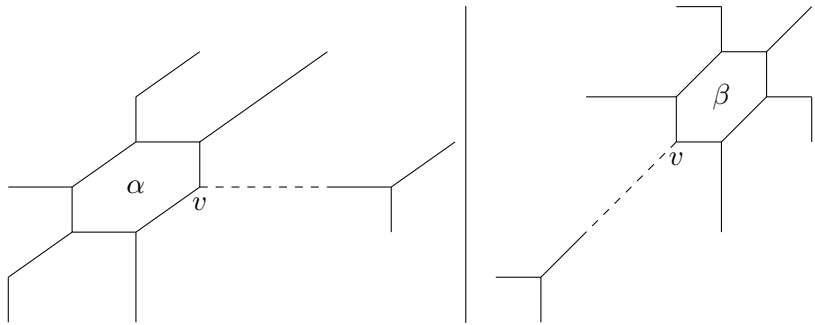

The following Proposition is included for completeness,
to show that we have a crossing-free tropical variety in Figure~\ref{fig:genus2roughDraft}.
If the reader is sufficiently convinced by the image in Figure~\ref{fig:genus2roughDraft},
they may prefer to continue reading the proof in Proposition~\ref{prop:genus2bumps}.

\begin{proposition}
  The rough draft in Figure~\ref{fig:genus2roughDraft} does not contain any crossings.
\end{proposition}
\begin{proof}
  To start: the two hexagons do not cross each other because they are separated by the plane $x + y = 0$.

  Second, the rays starting at the hexagons do not cross the hexagons.
These rays can all be separated by a plane that contains one of the edges of the hexagon at the vertex
where the ray originates.

  Also, the rays starting at the hexagons do not intersect other such rays.
We can see this in Figure~\ref{fig:genus2planes} or by writing down the rays.

  For example, the rays of the $\beta$ hexagon have $z = 0$ and
do not have a chance of intersecting most of the rays of the $\alpha$ hexagon.
If we extend the lines of the $\beta$ hexagon to infinity in Figure~\ref{fig:genus2planes},
they separate all the rays, including the one ray of the $\alpha$ hexagon.

  Lastly, we have all the infinite rays that branch off of another ray.
Let us first consider those rays in the direction $(-1,0,0)$.
Of course, none of these rays will intersect each other because they are parallel.

  Neither will they intersect the rays in the direction $(0,-1,0)$
since every ray in the direction $(-1,0,0)$ lies on one side of the plane $x = y$ and
every ray in the direction $(0,-1,0)$ on the other.

  Nor will they intersect the hexagons or the rays coming off of the hexagons
which we can see by examining the position of each of the rays with respect to the planes
$x = y$, $z = 0$ or $x + y = 0$. Figure~\ref{fig:genus2planes} gives some insight to this.

  For example, at $\alpha_2$, the ray in the $(-1,0,0)$ direction has $z > 0$ and $x \le y$.
So it will not intersect anything with $z \le 0$, nor anything with $x > y$, and
it only intersects the plane $x = y$ at one point.
This excludes everything.
In fact, by checking each ray,
we see that these three planes ($x = y$, $z = 0$ and $x + y = 0$) are enough to separate each ray.

  The rays in the direction $(0,-1,0)$ are just the mirror image of those in the direction $(-1,0,0)$
after reflecting in the $x = y$ plane.
So anything we said about the $(-1,0,0)$-rays holds for the $(0,-1,0)$ rays.

  The story is the similar for the rays in the direction $(1,1,1)$ and $(0,0,-1)$.
For example, rays in the direction $(1,1,1)$ all start with $z \ge 0$ and
rays in the direction $(0,0,-1)$ all start with $z \le 0$.
So these types of rays don't intersect each other, nor the hexagons,
nor the rays coming directly off of the hexagons.

  Finally, there is no intersection between infinite rays in any direction
as we can see by checking the position with respect to various planes at each ray.
Namely, the planes $x = y$, $z = 0$, $x + y = 0$ work.
\end{proof}

\begin{figure}[htbp]
  \centering
  \tdplotsetmaincoords{60}{20}
  \begin{tikzpicture}[tdplot_main_coords,scale=1.2]
    % alpha circle
    \draw (0,0,0) -- (0,0,1/5) -- (0,1/5,2/5) -- (0,1/5,3/5) -- (0,0,4/5) -- (0,0,1) -- (-1,-1,1) -- (-2,-2,0) -- (-2,-2,-1) -- (-1,-1,-1) -- (0,0,0);
    % beta circle
    \draw (0,0,0) -- (1,0,0) -- (6/5,1/5,0) -- (7/5,2/5,1/5) -- (8/5,3/5,1/5) -- (9/5,4/5,0) -- (2,1,0) -- (2,2,0) -- (1,2,0) -- (0,1,0) -- (0,0,0);
    % delta rays
    \draw (0,0,1/5) -- (0,-2,1/5);
    \draw (0,0,4/5) -- (0,-2,4/5);
    \draw (0,1/5,2/5) -- (0,2.5,2/5);
      \draw (0,2.5,2/5) -- (0,2.5,-1);
      \draw (0,2.5,2/5) -- (-1,2.5,2/5);
      \draw (0,2.5,2/5) -- (1,3.5,7/5);
    \draw (0,1/5,3/5) -- (0,3,3/5);
      \draw (0,3,3/5) -- (0,3,-1);
      \draw (0,3,3/5) -- (-1,3,3/5);
      \draw (0,3,3/5) -- (1,4,8/5);
    % gamma rays
    \draw (6/5,1/5,0) -- (6/5,1/5,-1);
    \draw (9/5,4/5,0) -- (9/5,4/5,-1);
    \draw (7/5,2/5,1/5) -- (7/5,2/5,3/5);
      \draw (7/5,2/5,3/5) -- (1,2/5,3/5);
      \draw (7/5,2/5,3/5) -- (7/5,0,3/5);
      \draw (7/5,2/5,3/5) -- (9/5,3/5,1);
    \draw (8/5,3/5,1/5) -- (8/5,3/5,3/5);
      \draw (8/5,3/5,3/5) -- (6/5,3/5,3/5);
      \draw (8/5,3/5,3/5) -- (8/5,1/5,3/5);
      \draw (8/5,3/5,3/5) -- (2,1,1);
  \end{tikzpicture}
  \caption{Position of the new rays added from the rough draft.}
  \label{fig:genus2all}
\end{figure}
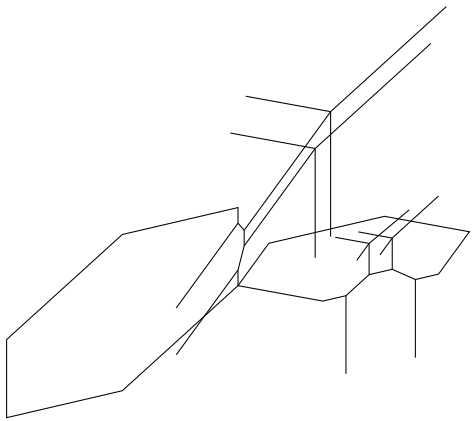

Now let us look at the construction in Table~\ref{tab:listofdivisors}
 which has some extra bits added to it, pictured in Figure~\ref{fig:genus2all}.
The bumps at $\delta_1,\dots,\delta_4$ and $\gamma_1,\dots,\gamma_4$ are small enough
that they should not impact injectivity.
But we can also make the bumps arbitrarily small if
we are concerned by decreasing the distances between $\delta_1$ and $\delta_2$ and
between $\gamma_1$ and $\gamma_2$.

\begin{figure}[htbp]
  \centering
  \tdplotsetmaincoords{90}{90} % x = 0
  \begin{tikzpicture}[tdplot_main_coords]
    % alpha circle
    \draw (0,0,0) -- (0,0,1/5) -- (0,1/5,2/5) -- (0,1/5,3/5) -- (0,0,4/5) -- (0,0,1) -- (-1,-1,1) -- (-2,-2,0) -- (-2,-2,-1) -- (-1,-1,-1) -- (0,0,0);
    % beta circle
    \draw (0,0,0) -- (1,0,0) -- (6/5,1/5,0) -- (7/5,2/5,1/5) -- (8/5,3/5,1/5) -- (9/5,4/5,0) -- (2,1,0) -- (2,2,0) -- (1,2,0) -- (0,1,0) -- (0,0,0);
    % alpha rays
    \draw (0,0,1) -- (1,1,2);
    \draw (-1,-1,1) -- (-1,-1,2);
    \draw (-2,-2,0) -- (-3,-3,0);
    \draw (-2,-2,-1) -- (-3,-3,-2);
    \draw (-1,-1,-1) -- (-1,-1,-2);
    % beta rays
    \draw (1,0,0) -- (1,-1,0);
    \draw (2,1,0) -- (3,1,0);
    \draw (2,2,0) -- (3,3,0);
    \draw (1,2,0) -- (1,3,0);
    \draw (0,1,0) -- (-1,1,0);
    % gamma rays
    \draw (0,0,1/5) -- (0,-1,1/5);
    \draw (0,1/5,2/5) -- (0,6/5,2/5);
    \draw (0,1/5,3/5) -- (0,6/5,3/5);
    \draw (0,0,4/5) -- (0,-1,4/5);
    % delta rays
    \draw (6/5,1/5,0) -- (6/5,1/5,-1);
    \draw (7/5,2/5,1/5) -- (7/5,2/5,6/5);
    \draw (8/5,3/5,1/5) -- (8/5,3/5,6/5);
    \draw (9/5,4/5,0) -- (9/5,4/5,-1);
    % starting points
    \foreach \x/\y/\z in { 1.4/0.4/1.2, 1.6/0.6/1.2}
      \filldraw (\x,\y,\z) circle (0.5 mm);
  \end{tikzpicture}
  \vrule
  \tdplotsetmaincoords{90}{0} % y = 0
  \begin{tikzpicture}[tdplot_main_coords]
    % alpha circle
    \draw (0,0,0) -- (0,0,1/5) -- (0,1/5,2/5) -- (0,1/5,3/5) -- (0,0,4/5) -- (0,0,1) -- (-1,-1,1) -- (-2,-2,0) -- (-2,-2,-1) -- (-1,-1,-1) -- (0,0,0);
    % beta circle
    \draw (0,0,0) -- (1,0,0) -- (6/5,1/5,0) -- (7/5,2/5,1/5) -- (8/5,3/5,1/5) -- (9/5,4/5,0) -- (2,1,0) -- (2,2,0) -- (1,2,0) -- (0,1,0) -- (0,0,0);
    % alpha rays
    \draw (0,0,1) -- (1,1,2);
    \draw (-1,-1,1) -- (-1,-1,2);
    \draw (-2,-2,0) -- (-3,-3,0);
    \draw (-2,-2,-1) -- (-3,-3,-2);
    \draw (-1,-1,-1) -- (-1,-1,-2);
    % beta rays
    \draw (1,0,0) -- (1,-1,0);
    \draw (2,1,0) -- (3,1,0);
    \draw (2,2,0) -- (3,3,0);
    \draw (1,2,0) -- (1,3,0);
    \draw (0,1,0) -- (-1,1,0);
    % gamma rays
    \draw (0,0,1/5) -- (0,-1,1/5);
    \draw (0,1/5,2/5) -- (0,6/5,2/5);
    \draw (0,1/5,3/5) -- (0,6/5,3/5);
    \draw (0,0,4/5) -- (0,-1,4/5);
    % delta rays
    \draw (6/5,1/5,0) -- (6/5,1/5,-1);
    \draw (7/5,2/5,1/5) -- (7/5,2/5,6/5);
    \draw (8/5,3/5,1/5) -- (8/5,3/5,6/5);
    \draw (9/5,4/5,0) -- (9/5,4/5,-1);
    % starting points
    \foreach \x/\y/\z in { 1.4/0.4/1.2, 1.6/0.6/1.2}
      \filldraw (\x,\y,\z) circle (0.5 mm);
  \end{tikzpicture}
  \caption{Where the rays in $-x$ and $-y$ direction originate at $\gamma_2, \gamma_3$. Projection onto the $x = 0$ and $y = 0$ planes.}
  \label{fig:genus2xrays}
\end{figure}
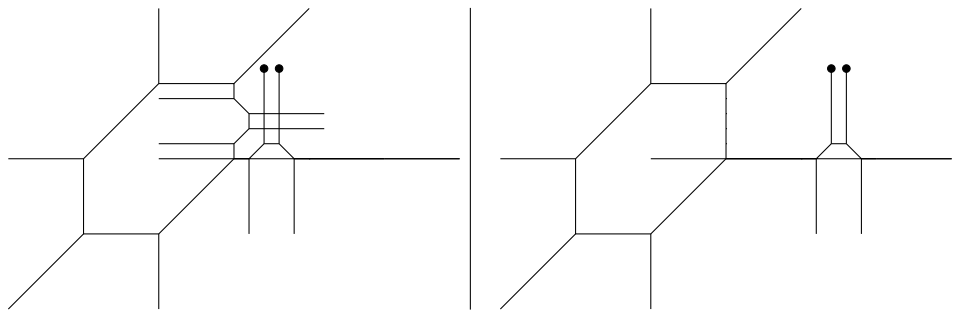

The idea, which one can see in Figure~\ref{fig:genus2xrays},
is that there is some compact set (possibly even finite) of lengths
that would cause an intersection and outside which, all other lengths work.

\begin{proposition} \label{prop:genus2bumps}
We can choose the lengths $\ell(\delta_2), \ell(\delta_3), \ell(\gamma_2), \ell(\gamma_3)$
to get an injective embedding of our curve.
\end{proposition}
\begin{proof}
First, project onto the plane $z = 0$.
Here you can see that the infinite rays at $\delta_1, \delta_4$ do not intersect any part of the rough draft.

Similarly, in the projection onto $x = 0$,
we can see that the rays at $\gamma_1, \gamma_4$ do not intersect any part of the rough draft.

Now, consider the finite rays at $\delta_2, \delta_3$.
These point in the $+y$ direction and, in fact, all other rays that point in the $+y$ direction are finite.
Meaning if $\ell(\delta_2)$ and $\ell(\delta_3)$ are large enough,
then there are no more rays parallel to the $\delta_2$ and $\delta_3$ rays.

  Therefore, after a certain threshold, when we start three infinite rays in the directions
$(-1,0,0), (0,0,-1)$ and $(1,1,1)$,
there are only finitely many lengths that would cause an intersection with any part of the rough draft.

  On the other hand, the ray at $\delta_3$ going in the $-z$ direction will intersect the finite ray at
$\delta_2$ if $\ell(\delta_3) \le \ell(\delta_2)$. If we assert that $\ell(\delta_2) < \ell(\delta_3)$,
there are no issues.

  For the rays at $\gamma_2, \gamma_3$, it is the same picture:
a bounded set of lengths that would cause an intersection,
afterwards the only issue is that the ray in the $(1,1,1)$ direction at $\gamma_2$
might intersect the finite ray at $\gamma_3$. So again, we assert that $\ell(\gamma_3) < \ell(\gamma_2)$.

By construction, the infinite rays have directions $(-1,0,0)$, $(0,-1,0)$, $(0,0,-1)$ or $(1,1,1)$.
It it easy to see that rays in these directions intersect at infinity in $\mb{TP}^3$ if and only if
they intersect in $\R^3$.
\end{proof}

Proposition~\ref{prop:genus2bumps} is the hard part.
Afterwards, smoothness and fully-faithfulness come for free from how we constructed the rough draft.

\begin{proposition}
If we choose the lengths $\ell(\delta_2), \ell(\delta_3), \ell(\gamma_2), \ell(\gamma_3)$
such that the tropicalization is injective it is also smooth and fully faithful.
\end{proposition}
\begin{proof}
  Since the map is injective, and along each edge the gcd of the slopes of the functions
$F_1, F_2, F_3$ is $1$ (by construction), thus the weight of every edge is $1$. Therefore, the map is fully-faithful.

For smoothness (which is also by construction), we simply have to check all the vertices.
For example, at $\alpha_1$ the outgoing directions are, according to Table \ref{tab:listofdivisors},
\begin{align*}
(1,1,1) \text{ along the ray towards infinity}, \\
(0,0,-1) \text{ along the ray towards } v, \\
(-1,-1,0) \text{ along the ray towards } \alpha_2.
\end{align*}
The lattice spanned by these three rays is $\{ (x,y, z) \in \Z^3 \mid x-y = 0 \}$.
This is clearly of rank $2$ and saturated.

At $v$, the rays are
\begin{align*}
(0,0,1) \text{ along the ray towards } \gamma_1, \\
(0,1,0) \text{ along the ray towards } \beta_1, \\
(1,0,0) \text{ along the ray towards } \beta_5, \\
(-1,-1,-1) \text{ along the ray towards } \alpha_5.
\end{align*}
The lattice spanned here is $\Z^3$.

All other vertices can be checked similarly.
Therefore, the tropicalization is smooth.
\end{proof}

\bibliography{Effective-Fully-Faithful}{}

\begin{thebibliography}{CDMY16}
\expandafter\ifx\csname url\endcsname\relax
  \def\url#1{\texttt{#1}}\fi
\expandafter\ifx\csname doi\endcsname\relax
  \def\doi#1{\burlalt{doi:#1}{http://dx.doi.org/#1}}\fi
\expandafter\ifx\csname urlprefix\endcsname\relax\def\urlprefix{URL }\fi
\expandafter\ifx\csname href\endcsname\relax
  \def\href#1#2{#2}\fi
\expandafter\ifx\csname burlalt\endcsname\relax
  \def\burlalt#1#2{\href{#2}{#1}}\fi

\bibitem[ABBR15]{ABBR}
Omid Amini, Matthew Baker, Erwan Brugallé, and Joseph Rabinoff.
\newblock {Lifting harmonic morphisms I: metrized complexes and Berkovich
  skeleta}.
\newblock {\em Research in the Mathematical Sciences}, 2(1):7, Jun 2015.
\newblock \doi{10.1186/s40687-014-0019-0}.

\bibitem[ABKS14]{ABKS}
Yang An, Matthew Baker, Greg Kuperberg, and Farbod Shokrieh.
\newblock Canonical representatives for divisor classes on tropical curves and
  the matrix-tree theorem.
\newblock {\em Forum Math. Sigma}, 2:e24, 25, 2014.
\newblock \doi{10.1017/fms.2014.25}.

\bibitem[Bak08]{Baker}
Matthew Baker.
\newblock Specialization of linear systems from curves to graphs.
\newblock {\em Algebra Number Theory}, 2(6):613--653, 2008.
\newblock \doi{10.2140/ant.2008.2.613}.
\newblock With an appendix by Brian Conrad.

\bibitem[Ber90]{Berk}
Vladimir~G. Berkovich.
\newblock {\em Spectral theory and analytic geometry over non-{A}rchimedean
  fields}, volume~33 of {\em Mathematical Surveys and Monographs}.
\newblock American Mathematical Society, Providence, RI, 1990.

\bibitem[BPR13]{BPR13}
Matthew Baker, Sam Payne, and Joseph Rabinoff.
\newblock On the structure of non-{A}rchimedean analytic curves.
\newblock In {\em Tropical and non-{A}rchimedean geometry}, volume 605 of {\em
  Contemp. Math.}, pages 93--121. Amer. Math. Soc., Providence, RI, 2013.
\newblock \doi{10.1090/conm/605/12113}.

\bibitem[BPR16]{BPR16}
Matthew Baker, Sam Payne, and Joseph Rabinoff.
\newblock Nonarchimedean geometry, trop\-ic\-al\-iz\-a\-tion, and metrics on
  curves.
\newblock {\em Algebr. Geom.}, 3(1):63--105, 2016.
\newblock \doi{10.14231/AG-2016-004}.

\bibitem[BR15]{BR}
Matthew Baker and Joseph Rabinoff.
\newblock The skeleton of the {J}acobian, the {J}acobian of the skeleton, and
  lifting meromorphic functions from tropical to algebraic curves.
\newblock {\em Int. Math. Res. Not. IMRN}, 2015(16):7436--7472, 2015.
\newblock \doi{10.1093/imrn/rnu168}.

\bibitem[CDMY16]{CDMY}
Dustin Cartwright, Andrew Dudzik, Madhusudan Manjunath, and Yuan Yao.
\newblock Embeddings and immersions of tropical curves.
\newblock {\em Collect. Math.}, 67(1):1--19, 2016.
\newblock \doi{10.1007/s13348-015-0149-8}.

\bibitem[CFPU16]{CFPU}
Man-Wai Cheung, Lorenzo Fantini, Jennifer Park, and Martin Ulirsch.
\newblock Faithful realizability of tropical curves.
\newblock {\em Int. Math. Res. Not. IMRN}, 2016(15):4706--4727, 2016.
\newblock \doi{10.1093/imrn/rnv269}.

\bibitem[Har77]{Har}
Robin Hartshorne.
\newblock {\em Algebraic geometry}.
\newblock Springer-Verlag, New York-Heidelberg, 1977.
\newblock Graduate Texts in Mathematics, No. 52.

\bibitem[Jel20]{J}
Philipp Jell.
\newblock Constructing smooth and fully faithful tropicalizations for {M}umford
  curves.
\newblock {\em Selecta Math. (N.S.)}, 26(4):Paper No. 60, 23, 2020.
\newblock \doi{10.1007/s00029-020-00586-2}.

\bibitem[MZ08]{MZ}
Grigory Mikhalkin and Ilia Zharkov.
\newblock Tropical curves, their {J}acobians and theta functions.
\newblock In {\em {Curves and Abelian Varieties}}, volume 465 of {\em Contemp.
  Math.}, pages 203--230. Amer. Math. Soc., Providence, RI, 2008.
\newblock \doi{10.1090/conm/465/09104}.

\bibitem[Pay09]{Payne}
Sam Payne.
\newblock Analytification is the limit of all tropicalizations.
\newblock {\em Math. Res. Lett.}, 16(3):543--556, 2009.
\newblock \doi{10.4310/MRL.2009.v16.n3.a13}.

\bibitem[Thu05]{Thu}
Amaury Thuillier.
\newblock {\em {Théorie du potentiel sur les courbes en géométrie analytique
  non archimédienne. Applications à la théorie d'Arakelov.}}
\newblock Thèse, {Universit{\'e} Rennes 1}, October 2005.
\newblock \urlprefix\url{https://tel.archives-ouvertes.fr/tel-00010990}.

\bibitem[Wag17]{Wag}
Till Wagner.
\newblock Faithful tropicalization of {Mumford} curves of genus two.
\newblock {\em Beitr{\"a}ge zur Algebra und Geometrie / Contributions to
  Algebra and Geometry}, 58(1):47--67, Mar 2017.
\newblock \doi{10.1007/s13366-015-0272-4}.

\end{thebibliography}
\bibliographystyle{halpha}
\end{document}